\documentclass[12pt,a4paper,twoside,english]{smfart}

\usepackage{lmodern}
\usepackage[utf8]{inputenc}
\usepackage[francais]{babel}
\usepackage[T1]{fontenc}
\usepackage{amssymb}
\usepackage{amsmath}
\usepackage{amsthm}
\usepackage{amsfonts}
\usepackage{graphicx}
\usepackage[colorlinks=true,linkcolor=black,citecolor=black,urlcolor=black]{hyperref}
\usepackage{fancyhdr}
\usepackage[all,cmtip]{xy}
\usepackage{alltt}
\usepackage{footmisc}
\usepackage{smfthm}

\usepackage{calrsfs}

\usepackage{fullpage}

\xyoption{all}

\xyoption{frame}

\newtheorem{Theorem}{Theorem}
\newtheorem*{TheoremSN}{Theorem}

\newtheorem*{conjecture}{Conjecture}
  \renewcommand{\theTheorem}{\Alph{Theorem}}

\def\Mapsto{%
  \xrightarrow[\raisebox{0.25 em}{\smash{\ensuremath{\sim}}}]{}%
}

\DeclareUnicodeCharacter{00A0}{~}

\def\Q{{\mathbb Q}}
\def\Z{{\mathbb Z}}
\def\fq{{\mathbb F}}

\def\Nat{{\mathbb N}}

\def\p{{\mathfrak p}}
\def\P{{\mathfrak P}}
\def\Ell{{\mathfrak L}}

\def\N{{\rm N}}
\def\R{{\rm R}}

\def\GG{{\rm G}}
\def\K{{\rm K}}
\def\L{{\rm L}}
\def\F{{\rm F}}
\def\k{{\rm k}}
\def\MM{{\rm M}}

\def\GG{{\rm G}}
\def\U{{\rm U}}
\def\C{{\rm C}}
\def\V{{\rm V}}
\def\I{{\rm I}}
\def\A{{\rm A}}

\def\H{{\mathcal H}}
\def\G{{\mathcal G}}

\def\O{{\mathcal O}}

\def\E{{\mathcal E}}
\def\M{{\mathcal M}}
\def\SS{{\mathcal S}}

\def\Ll{{\mathcal L}}

\def\Gl{{\rm Gl}}

\def\Im{{\rm Im}}
\def\Gal{{\rm Gal}}

\def\Aut{{\rm Aut}}

\def\Cl{{\rm Cl}}

\def\GST{{{\rm G}_S^T}}

\def\Reg{{\rm Reg}}
\def\Norm{{\rm Norm}}
\def\ker{{\rm ker}}
\def\Ind{{\rm Ind}}
\def\Hom{{\rm Hom}}
\def\Sl{{\rm Sl}}
\def\Tor{{\rm Tor}}

\def\pel{p, {\rm el}}

\def\gamchapzero{{\Gamma_\sigma^{\circ}}}
\def\gamchap{{\Gamma_\sigma}}
\def\gamchappel{{\Gamma_\sigma^{p,el}}}

\def\fpm{{fixed-point-mixing modulo Frattini}}
\def\FPM{{FPMF}}

\def\sl{{\mathfrak s} {\mathfrak l}}
\def\1{{\bf 1}}

\def\r#1{r_{\Delta}(#1)}

\def\FF#1#2{{\displaystyle{\big(\frac{#1}{#2}\big) }}}

\begin{document}

\date{\today}

\author{Farshid Hajir}
\address{Department of Mathematics \& Statistics, University of Massachusetts, Amherst MA 01003, USA.}

\author{Christian Maire}

\address{Laboratoire de Mathématiques de Besançon, Université Bourgogne Franche-Comté et CNRS (UMR 6623), 16 route de Gray, 25030 Besançon cédex, France.}

\title{Analytic  Lie  extensions of number fields with  cyclic fixed points and tame ramification}


\thanks{\emph{Acknowledgements.} The second author would like to thank the Department of Mathematics \& Statistics at UMass Amherst for its hospitality during several visits.   }

\begin{abstract}
 Let $p$ be a prime number and $\K$ an algebraic number field. What is the arithmetic structure of Galois extensions $\L/\K$ having $p$-adic analytic Galois group $\Gamma=\Gal(\L/\K)$?  The celebrated Tame Fontaine-Mazur conjecture predicts that such extensions are either deeply ramified (at some prime dividing $p$) or ramified at an infinite number of primes.  In this work, we take up a study (initiated by Boston) of this type of question under the assumption  that $\L$ is Galois over some subfield $\k$ of $\K$ such that $[\K:\k]$ is a prime $\ell\neq p$.  Letting $\sigma$ be a generator of $\Gal(\K/\k)$, we study the constraints posed on the arithmetic of $\L/\K$ by the cyclic action of $\sigma$ on $\Gamma$, focusing on the critical role played by the fixed points of this action, and their relation to the ramification in $\L/\K$. The method of Boston works only when there are no non-trivial fixed points for this action. We show that even in the presence of arbitrarily many fixed points, the action of $\sigma$ places severe arithmetic conditions on the existence of finitely and tamely ramified uniform $p$-adic analytic extensions over $\K$, which in some instances leads us to be able to deduce the non-existence of such extensions over $\K$ from their non-existence over $\k$. 
\end{abstract}

\maketitle

\tableofcontents

\section{Introduction} \label{section-introduction}

\subsection{Background} \label{rappels-Boston0}

\theTheorem

Fix a prime $p$. The theory of pro-$p$ groups has seen major advances in the last few decades. In particular, the monumental work \cite{lazard} of Lazard  on \emph{$p$-adic analytic} pro-$p$ groups (that is to say Lie groups over the field $\Q_p$ of $p$-adic numbers) has been simplified and reinterpreted in the book \cite{DSMN} by Dixon, du Sautoy, Mann, and Segal and has made the subject more readily applied in many situations and much more accessible to a variety of non-experts. At the same time, the theory of Galois representations encodes vast amounts of arithmetic information via action of Galois groups on finite-dimensional $p$-adic vector spaces, which is to say creates continuous homomorphisms from Galois groups to the $p$-adic Lie groups $\Gl_n(\Q_p)$. In this paper, we are interested in using group-theoretical information to derive consequences for \emph{finitely and tamely} ramified Galois representations.

\

  We recall that a pro-$p$ group $\Gamma$ is called \emph{uniform} if $\Gamma^p=\langle x^p, \ x \in \Gamma \rangle$ contains the commutators $[\Gamma,\Gamma]$  of  $\Gamma$ and if moreover $\Gamma$ is torsion-free.  By Lazard \cite{lazard} (see also \cite{DSMN}), every finite-dimensional $p$-adic analytic group (closed subgroup of $\Gl_n(\Q_p)$ for some $n\geq 1$) has a finite-index (open) {uniform} subgroup.

In \cite{Boston2} and \cite{Boston3}, Boston initiated the study of the following situation (see also Wingberg \cite{Wingberg} and Maire \cite{Maire-MRL}).
We fix a uniform pro-$p$ group $\Gamma$ and assume that $\Gamma$ is realized as the
  Galois group of a \emph{tamely ramified} extension $\L/\K$, i.e. $\Gamma=\Gal(\L/\K)$, and we assume, moreover, that $\Gamma$ is equipped with a semi-simple Galois action. To be more explicit, from now on we assume that:
\begin{itemize}
\item $\K$ is a finite Galois extension of a number field $\k$ with Galois group $\Delta=\Gal(\K/\k)$ 
\item $\Delta$ is a cyclic group of prime order $\ell$ dividing $p-1$, and we fix a generator $\sigma$ of $\Delta$
\item $\L/\K$ is a finitely and tamely ramified Galois extension which is Galois over $\k$
\item $\Gamma=\Gal(\L/\K)$ is a uniform pro-$p$ group of finite dimension $d$
\end{itemize}

   \begin{TheoremSN}[Boston] 
  Under the above assumptions, if in addition
  \begin{itemize}
  \item $p$ does not divide the order of the class group of $\k$, and 
  \item $\L/\K$ is everywhere unramified,
  \end{itemize}
 then $\Gamma$ is trivial.
 \end{TheoremSN}
 
Here's the strategy of Boston's proof of this result.  The assumptions made in the theorem imply that $\sigma$ acts without non-trivial fixed points on $\Gamma^{\textrm{ab}}$ (to simplify the terminology, we say by way of shorthand that the action of $\sigma$ is "FPF (fixed-point-free)"). By the uniformity of $\Gamma$ the action of $\sigma$ is fixed-point-free also on $\Gamma$.  The existence of this fixed-point-free cyclic action on $\Gamma$ implies that $\Gamma$ is nilpotent (see Proposition \ref{FPF}). We recall that a group is called FAb if for all open subgroups $U$, the abelianization $U^{\textrm{ab}}$ is finite. Since $\L/\K$ is tamely ramified, $\Gamma$ is FAb.    Since $\Gamma$ is both nilpotent and FAb, it is finite; but as a uniform group, it is torsion-free, hence must be trivial. 

\medskip

In this work, we attempt to extend Boston's strategy to the case of (tamely) ramified $\L/\K$.  The key challenge is to handle the fixed points introduced by ramification because Boston's proof relies heavily on the fact the $\sigma$-action in the unramified case is fixed-point-free. We refer to \cite{Hajir-Maire} for a different application of this phenomenon in the context of Iwasawa theory where one allows wild ramification in $\L/\K$.

\subsection{A sample result} \label{rappel-boston}
In order to state our results, we need to introduce some more notation and hypotheses.  
Let $S$ be a finite set of places of $\K$ all of which are prime to $p$ (we say that the set $S$ is tame and indicate this by writing $(S,p)=1$). Since we will be working $p$-extensions in which the primes in $S$ are allowed to ramify, we further assume that for finite places  $\p\in S$, we have  $\# \O_\K/\p \equiv 1 (\bmod \ p).$  We let $\K_S$ be the maximal pro-$p$ extension of $\K$ unramified outside $S$ and we put $\GG_S=\GG_S(\K)=\Gal(\K_S/\K)$.

\medskip

Let us also take an auxiliary finite set $T$ of places of $\K$, disjoint from $S$, and define $\K_S^T$ to be the maximal pro-$p$ extension of $\K$ unramified outside $S$ and in which the places in $T$ split completely.  We put  $\GG_S^T=\GG_S^T(\K)=\Gal(\K_S^T/\K)$.  We note then that $\K_S^T \subset \K_S$, that $\GG_S \twoheadrightarrow \GG_S^T$ and that $\K_S^\emptyset= \K_S$.

\medskip

Recall that $\K$ is a number field admitting a non-trivial automorphism $\sigma$ of prime order $\ell$ dividing $p-1$, and $\k=\K^\sigma$ is the fixed field of $\Delta=\langle \sigma \rangle$.  We will assume that the sets $S$ and $T$ described above are stable under the action of $\sigma$.  Thus, the extension  $\K_S^T/\k$ is Galois and $\sigma$ acts on $\GG_S^T=\Gal(\K_S^T/\K)$.

\begin{defi}
Consider a continuous Galois representation $\rho : \GG_S^T(\K) \rightarrow \Gl_n(\Q_p)$, and let $\L$ be the subfield of $\K_S^T$ fixed by $\ker(\rho)$ so that the image $\Gamma$ of $\rho$ is naturally identified with 
$\Gal(\L,\K)$.  We say that $\rho$ (or $\Gamma$) is \emph{$\sigma$-uniform} if
we have (i)  $\Gamma=\Gal(\L/\K)$ is uniform; and (ii) $\L/\k$ is Galois, i.e. the action of $\sigma$ on $\GG_S^T(\K)$ induces an action on $\Gamma$.  
\end{defi}

For a finitely generated pro-$p$ group $G$, recall that closed subgroup generated by $p$th powers and commutators, $\Phi(G)=G^p[G,G]$, is the \emph{Frattini subgroup} of $G$; it is a characteristic subgroup of finite index.  The \emph{Frattini quotient} ${G}^{\pel}:=G/\Phi(G)$ is the maximal abelian exponent $p$ quotient of $G$.
The method of Boston described in \S \ref{rappels-Boston0} in the unramified case carries over to $\GG_S^T$ without any trouble only if the action of $\sigma$ on $\Gamma$ is FPF.  More precisely, if the action of $\sigma$ on  $\GG_S^T/\Phi(\GG_S^T)$ is fixed-point-free, then any $\sigma$-uniform representation of $\GG_S^T$ has trivial image. As indicated above, we try to extend the method by introducing fixed points that result from allowing tame ramification.  We show that even in the presence of non-trivial fixed points, all $\sigma$-uniform quotients of $\GG_S^T$ are trivial as long as the "new" ramification is restricted to the subgroup generated by the fixed points.   In \S \ref{section-presentation}, we will present our results in greater generality, but we first illustrate them by presenting a special case for the well-known uniform and FAb pro-$p$ group $\Sl_2^1(\Z_p):=\ker \big( \Sl_2(\Z_p) \rightarrow \Sl_2(\fq_p)\big)$  of dimension~$3$.

\begin{TheoremSN}
Suppose $\K/\k$ is a quadratic extension with Galois group $\Delta=\langle \sigma \rangle$ such that the odd prime $p$ does not divide the class number of $\k$. Let $\Gamma= \Sl_2^1(\Z_p)$. Suppose for all finite sets $\Sigma$ of places of $\k$ with $(\Sigma,p)=1$, there is no continuous Galois representation  $G_\Sigma(\k) \twoheadrightarrow \Gamma$. Then 
there exist infinitely many disjoint finite sets $S$ and $T$ of primes of $\K$, with $(S,p)=1$ and $|S|$ arbitrarily large, such that 
\begin{enumerate}
\item[(i)] $\GG_S^T(\K)$ is infinite,
\item[(ii)]  ${\GG_S^T(\K)}^{\pel}$ has $|S|$ independent fixed points under the action of $\sigma$,
\item[(iii)] there is no continuous $\sigma$-uniform representation 
$\rho:\GG_S^T(\K) \twoheadrightarrow \Gamma$.
\end{enumerate}
\end{TheoremSN}

\begin{rema}
The above theorem holds with $\Gamma=\Sl_n^1(\Z_p)$ for arbitrary $n$, under the additional assumption that the action of the automorphism $\sigma$ on $\Gamma$ corresponds to conjugation by a matrix of order~$2$ in $\Gl_n(\Z_p)$. For more details, see \S \ref{sln}.
\end{rema}

\subsection{Motivation}\label{FontMaz}
An important and vast "modularity" conjecture forms the motivation for the study begun by Boston and continued here. 
In \cite{FM}, Fontaine and Mazur propose a characterization of all Galois representations which arise from the action of the absolute Galois group of $\K$ on Tate twists of \'etale cohomology groups of algebraic varieties defined over $\K$: namely they predict that these are precisely the representations which are ramified at a finite number of primes of $\K$ and are potentially semistable at the primes dividing $p$. If we
restrict our attention to $p$-adic representations which are finitely and tamely ramified, we obtain the following consequence (Conjecture 5a of \cite{FM}) of this characterization (see Kisin-Wortmann \cite{Kisin-Wortmann} for the details). 

\begin{conjecture}[Tame Fontaine-Mazur Conjecture]\label{conjectureTFM}
For  a finite set $S$ of primes of $\K$ of residue characteristic not equal to $p$, and $n\geq 1$, any continuous Galois representation $\rho:\GG_S \longrightarrow \Gl_n(\Q_p)$ has finite image. 
\end{conjecture}

The philosophy of this conjecture rests on the idea that the eigenvalues of Frobenius (under a finitely and tamely ramified $p$-adic representation $\rho$) ought to be roots of  unity.  Consequently, the image of such a representation is solvable, and hence finite by class field theory (because it is also FAb).  We refer the reader to  \cite{Kisin-Wortmann} for further details.

\

One immediately checks the Conjecture for $n=1$ by Class Field Theory.  For $n>1$, on the other hand,  the Tame Fontaine-Mazur Conjecture in general appears to be completely out of reach, and the evidence for it for $n>2$ is rather preliminary. However, for $\K=\Q$, and $n=2$, the pioneering methods of Wiles and Taylor-Wiles can be used to show that many types of 2-dimensional representations do come from algebraic geometry (in fact from weight one modular forms) and hence have finite image.  As a partial list of such results, we refer the reader to Buzzard-Taylor \cite{BT},  Buzzard \cite{Buzzard}, Kessaei \cite{Kassaei1}, Kisin \cite{Kisin}, Pilloni \cite{Pilloni}, Pilloni-Stroh \cite{Pilloni-Stroh}. 

\

Recalling that every finitely generated $p$-adic analytic group has a uniform open subgroup, and that a uniform group of dimension $1$ or $2$ has quotient isomorphic to $\Z_p$,  the Tame Fontaine-Mazur conjecture can be rephrased as follows.

\begin{conjecture}[Tame Fontaine-Mazur Conjecture -- Uniform Version]\label{conjectureFM} 
Suppose $\K$ is a number field, and $\Gamma$ is a uniform pro-$p$ group of dimension $d>2$, hence infinite.  Then there does not exist a finitely and tamely ramified Galois extension $\L/\K$ with Galois group $\Gamma=\Gal(\L/\K)$. 
\end{conjecture}

In the simplest non-trivial case, one can take $\K=\Q$ and $\Gamma=\Sl_2^1(\Z_p)$.  We must then show that  $\Sl_2^1(\Z_p)$ cannot be realized as the Galois group of a finitely and tamely ramified Galois extension over $\Q$.   Given the recent spectacular breakthroughs listed above, perhaps the current methods will one day prove sufficient to establish this special case of the Tame Fontaine-Mazur conjecture, but at the moment the theory of even Galois representations is still under-developed by comparison with odd ones. We should emphasize that in this work, we rely exclusively on group-theoretical methods.  However, as automorphic methods approach a full proof of the tame Fontaine-Mazur conjecture (for 2-dimensional representations at least) over $\Q$, one can use the group-theoretical techniques discussed here to deduce some cases of the Tame Fontaine-Mazur conjecture over quadratic fields from known cases over $\Q$.

\medskip


\section{Presentation of  results} \label{section-presentation}

\subsection{A key definition}\label{fpm} Recall that $\Gamma$ is a uniform pro-$p$ group equipped with the action of an automorphism $\sigma$ of prime order $\ell ~|~ p - 1$. 
We denote by  $$\gamchapzero=\langle \gamma \in \Gamma, \ \sigma(\gamma)=\gamma \rangle,$$ the closed subgroup of $\Gamma$ generated by the fixed points of $\Gamma$ under the action of $\sigma$, and let $\gamchap$ be its normal closure  in $\Gamma$. Let $G:=\Gamma/\gamchap$.

\begin{defi} \label{defi-fpm}
With the above assumptions, the action of $\sigma$ on $\Gamma$ is said to be \emph{\fpm} (\emph{\FPM}) if $G=\Gamma/\Gamma_\sigma$ acts non-trivially on $\gamchap/\Phi(\gamchap)$. 
\end{defi}

This notion will be essential for our work; its relevance is explained at the end of \S \ref{strat-proof}. Let us give two examples that we will study in section \ref{section-examples} and will be important to illustrate our results.

\begin{exem}[See \S \ref{sl2}] If a FAb and  uniform pro-$p$ group  of dimension $3$ admits non-trivial action by  an automorphism $\sigma$ of order $2$, then this action  is \fpm.
Thus, any involution which acts non-trivially on the linear group $\Sl_2^1(\Z_p):=\ker\big(\Sl_2(\Z_p) \rightarrow \Sl_2(\fq_p)\big)$  is \fpm.
\end{exem}

\begin{exem}[See \S \ref{sln}]
More generally, for the FAb pro-$p$ group $$\Sl_n^1(\Z_p):=\ker\big(\Sl_n(\Z_p)\rightarrow \Sl_n(\fq_p)\big)\qquad n\geq 2,$$ and the automorphism  $\sigma_A$ coming from conjugation by a matrix $A\in \Gl_n(\Z_p)$ of order~$2$, the action of $\sigma_A$ is \fpm.
\end{exem}

\subsection{When $\sigma$ is of order $2$}

The case where the automorphism $\sigma$ is an involution, i.e. $\ell=2$, is particularly interesting.  Let us begin with a definition.

\begin{defi}
Let $\Gamma$ be a uniform group of dimension $d$, which also then equals the $p$-rank of $\Gamma$, i.e.  $\Gamma/\Phi(\Gamma)$ is a $d$-dimensional vector space over $\fq_p$. Suppose 
 $\sigma \in \Aut(\Gamma)$ has order $2$.  If  the multiplicity of the trivial character in the action of $\sigma$ on $\Gamma/\Phi(\Gamma)$ is $r$, we say that the action of $\sigma$ on $\Gamma$ is of type $(r,d-r)$ and write $t_{\sigma}(\Gamma)=(r,d-r)$.  We will say that the type of action of automorphisms of order $2$ on $\Gamma$ is constant if for all  $\sigma, \tau \in \Aut(\Gamma)$ of order $2$, $t_\sigma(\Gamma)=t_\tau(\Gamma)$. 
\end{defi}

\medskip

Under our blanket assumption that $\sigma$ is non-trivial, it is easy to see that $t_\sigma(\Gamma)\neq (d,0)$.  In \cite{Boston2} and \cite{Boston3}, the assumption is always that $t_\sigma(\Gamma)=(0,d)$.   In this work, we consider the more general intermediate types $t_\sigma(\Gamma)=(r,d-r)$ with $0< r < d$, by allowing tame ramification.

 \medskip
 
 The result we want to present will involve the Hilbert $p$-class field $\K^H$ of $\K$ so we recall this concept. Recalling that the prime $p$ has been throughout fixed, we let $\Cl(\K)$ be the $p$-Sylow subgroup of the ideal class group of $\K$ and $\K^H$ the maximal abelian unramified $p$-extension of $\K$.   The Artin map gives a canonical isomorphism $\Cl(\K) \to \Gal(\K^H/\K)$.  More generally, if $S$ is a finite tame set of places of $\K$ and $T$ is another finite set of places disjoint from $S$, $\Cl_S^T$ will be the $p$-Sylow subgroup of the $T$-ray class group of $\K$ mod $S$, which corresponds via the Artin map to $\Gal(\K_S^T/\K)$.
 
 As before, put $\GG_S^T=\GG_S^T(\K)=\Gal(\K_S^T/\K)$.

\begin{Theorem} \label{maintheo1bis} Let $p>2$ and let $s \in \Nat$.
Let $\K/\k$ be a quadratic extension and suppose that  $p$ does not divide $|\Cl(\k)|$.
Let $T$ be a finite set of places of $\k$ totally split in $\K^H/\K$ of large enough cardinality (see Theorem \ref{constante-liberte} for a more exact statement), and such that $\Cl^T(\K^H)$ is trivial.
Then there exist $s$ pairwise disjoint positive-density sets $\SS_i$, $i=1,\cdots,s$ of prime ideals $\p\subset \O_\k$ of $\k$ such that for finite sets  $S=\{\p_1,\cdots, \p_s\}$, with $\p_i \in \SS_i$, we have
\begin{enumerate}
\item[(i)] under the action of $\sigma$, there are $s$ independent fixed points in $\GG_S^T/\Phi(\GG_S^T)$;
\item[(ii)] there is no continuous representation $\rho : \GG_S^T \rightarrow \Gl_m(\Q_p)$ with $\sigma$-uniform image $\Gamma$ which is \fpm.
\end{enumerate}
\end{Theorem} 

\begin{rema}
The key point for obtaining (ii) above is as follows. The choices of $S$ and $T$ are made so that $(\GG_S)^{\pel}=(\GG_S^T)^{\pel}$. The action of $\sigma$ on $\GG_S^T$ has type $(s,d_p \Cl(\K))$, so we can rule out the existence of $\rho$ with $\sigma$-uniform image $\Gamma$ if $t_\sigma(\Gamma)$ is not compatible with $t_\sigma(\GG_S^T)$. Such incompatibility can be caused at the level of the subgroup generated by the fixed points of $\sigma$.   Typically, condition (i) of Theorem 
 \ref{maintheo1bis} assures an incompatibility for certain automorphisms of order $2$.  We will see that when $\sigma$ has order $2$, the contradiction can be detected already at the level of the Hilbert $p$-class field of $\K$.
\end{rema}

When $\Gamma$ has constant type for all order 2 automorphisms, the following interesting situation arises.

\begin{coro}\label{coro1} 
If the uniform group $\Gamma$ of dimension $d>0$ is such that for all $\sigma \in \Aut(\Gamma)$ of order $2$, $\sigma$ is \fpm, then under the conditions of Theorem \ref{maintheo1bis}, all continuous representations 
 $\rho : \GG_S^T \rightarrow \Gl_m(\Q_p)$ with  $\sigma$-uniform image $\Gamma$ come from $\k$. In particular, such a representation does not exist if either
 \begin{enumerate}
\item[-]The Tame Fontaine-Mazur conjecture holds for  $\k$, or
\item[-]  $d > |S|$.
\end{enumerate}
\end{coro}

\medskip

We can apply Theorem \ref{maintheo1bis} to the groups $\Sl_n^1(\Z_p)=\ker\big(\Sl_n(\Z_p) \rightarrow \Sl_n(\fq_p)\big)$, $n\geq 2$.  For all $n\geq 2$, $\Sl_n^1(\Z_p)$ is a uniform FAb group of dimension
$n^2-1$. We consider automorphisms $\sigma_A$ of order $2$ obtained via conjugation by a diagonalizable matrix
$A \in \Gl_n(\Z_p)$.  According to \cite{Seligman}, all automorphisms $\sigma$ which act trivially on the Cartan subalgebra are of the form $\sigma_A$ for some $A \in \Gl_n(\Z_p)$. We thus obtain the theorem of \S \ref{rappel-boston} and the remark following it.

\begin{coro}\label{coro-sln} Under the conditions of Theorem \ref{maintheo1bis}, there exist $s$ positive-density sets $\SS_i$, $i=1,\ldots, s$ of prime ideals  $\p\subset \O_\k$ of $\k$, such that for all finite sets 
$S=\{\p_1,\cdots, \p_s\}$,  with  $\p_i \in \SS_i$, and all $n\geq 2$, there does not exist a continuous representation $\rho : \GG_S^T \rightarrow \Gl_m(\Q_p)$ with $\sigma$-uniform image $\Sl_n^1(\Z_p)$ where the involution $\sigma=\sigma_A$ is conjugation by a diagonalizable matrix  $A \in \Gl_n(\Z_p)$.
\end{coro}

\subsection{When $\sigma$ is of order $\ell  ~|~ p-1$.}

The results of the previous section for involutions can be generalized for other automorphisms.  When $\sigma$ has order $\ell >2$, we need to introduce  a notion of ramification in relation to the normal subgroup
$\gamchap$ of $\Gamma$.

\begin{defi} Let $\rho : \GG_\Sigma^T \rightarrow \Gl_m(\Q_p)$ with image $\Gamma$ be a $\sigma$-uniform representation for some $\sigma\in \Aut(\Gamma)$ of order prime to $p$.  For a subset $S' \subset \Sigma$, we say that $\gamchap$ is supported at $S'$ if the 
inertia groups at the places in $S'$ generate $\gamchap$.
\end{defi}

\medskip

For a positive integer $n$, let  $\K_S^{(n)}/\K$  be the $n$th stage of the $p$-tower $\K_S/\K$, i.e. $\K_S^{(n)}$ is the fixed field of the $n$th term of the central series of $\GG_S=\GG_S(\K)$. Since $S$ does not
contain primes dividing $p$, by class field theory, $\K_S^{(n)}/\K$ is a finite extension.

\medskip

For a prime $\ell>3$, we put
 \begin{eqnarray}\label{entier-m}\displaystyle{m(\ell)=1+ \lceil 2^{\ell-1}\log_2(\ell -1) -\log_2(\ell-1)(\ell-2) \rceil} ;\end{eqnarray} 
we also define
  $m(2)=1$ and $m(3)=2$ -- see remark \ref{longueur-derivee} for an explanation of the motivation for this definition.
 
 \medskip

\begin{Theorem}\label{maintheorem0}  Let $\K$ be a number field admitting an automorphism $\sigma$ of order $\ell ~|~ p-1$ with fixed field $\k=\K^\sigma$.
Let $S$ be a finite set of primes of $\k$ not dividing $p$ such that the action of  $\sigma$ on $\GG_S^{{ab}}$ has no non-trivial fixed-points.
Let $s \in \Z_{>0}$.  Then there is an integer $A$ depending only on $[\K^{(m(\ell))}_S:\K]$,  $s$ and $|S|$ such that if $T$ is a finite set of places of $\k$ that split completely in $\K^{(m(\ell))}_S/\k$ satisfying $|T|\geq \A$, then there exist $s$ positive-density sets $\SS_i$, $i=1,\cdots,s$, of prime ideals $\p\subset \O_\k$ of $\k$, with the property that for all finite sets $S'=\{\p_1,\cdots, \p_s\}$, with $\p_i \in \SS_i$, we have:
\begin{enumerate}
\item[(i)] the action of $\sigma$ on $\GG_\Sigma^{T}/\Phi(\GG_\Sigma^T)$ has $s$ independent fixed points, where $\Sigma=S\cup S'$;
\item[(ii)] there is no FPMF $\sigma$-uniform continuous representation $\rho : \GG_\Sigma^T \rightarrow \Gl_m(\Q_p)$ such that $\gamchap$ is  supported at $S'$, where $\Gamma$ is the image of $\rho$.
\end{enumerate} 
\end{Theorem}

\subsection{Along the cyclotomic $\Z_p$-extension}\label{section-iwasawa}

The realm of $\Z_p$-extensions is particularly rich for providing situations where we can take $T$ (in the consideration of previous sections) as small as possible.  Let $\displaystyle{\k_\infty=\bigcup_n \k_n}$ (resp. $\displaystyle{\K_\infty=\bigcup_n \K_n}$) be the cyclotomic $\Z_p$-extension of $\k$ (resp. of $\K$), where as before $\K/\k$ is a cyclic degree $\ell ~|~ p-1$ extension with Galois group $\langle \sigma \rangle$. Let us assume that all along $\k_\infty/\k$, the $p$-class groups of the fields $\k_n$ are trivial.  Then $\sigma$ acts without fixed-points on the $p$-class groups $\Cl(\K_n)$ of the fields $\K_n$.  On the other hand, the growth of $[\K_n:\K]=p^n$ allows us to apply Theorem \ref{maintheorem0} with $T=\emptyset$ as long as $n$ is sufficiently large (when $\ell=2$, we have to assume in addition that the real infinite places of $\k$ do not complexify in $\K/\k$).

\begin{Theorem} \label{cyclo} Let $s \in \Z_{>0}$. Let
$\K/\Q$ be a real quadratic field and $\Gal(\K/\Q)=\langle \sigma \rangle$. Let $\displaystyle{\K_\infty=\bigcup_n\K_n}$ (resp. $\displaystyle{\Q_\infty=\bigcup_n\Q_n}$) be the cyclotomic $\Z_p$-extension of $\K$ (resp. of $\Q$). We assume that Greenberg Conjecture holds for $\K$ (the invariants $\mu$ and $\lambda$ vanish).
Then for $n_0\gg 0$, there exist  $s$ positive density sets $\SS_i$, $i=1,\cdots, s$, of places  $\p \subset \O_{\Q_{n_0}}$, such that for all finite sets $S =\{\p_1,\cdots, \p_s\}$, with $\p_i \in \SS_i$, and for all $n\geq n_0$,
\begin{enumerate}
\item[(i)] the action of $\sigma$ on $\GG_S(\K_n)/\Phi(\GG_S(\K_n))$ has $s$ independent fixed points;
\item[(ii)] there is no FPMF continuous $\sigma$-uniform $\rho : \GG_S(\K_n) \longrightarrow \Gl_m(\Q_p)$ such that $\gamchap$ is  supported at $S$, where $\Gamma$ is the image of $\rho$.
\end{enumerate} 
\end{Theorem}

 \begin{rema}
 We note that the pro-$p$ group $\GG_\emptyset(\K)$ is infinite as soon as the $p$-rank of $\Cl(\K)$ is at least $3$ (Schoof \cite{Schoof}). 
 \end{rema}

\subsection{Strategy of the proofs and outline of the rest of the paper}\label{strat-proof}

Our main results combine a number of ingredients: the effect of a semisimple cyclic action with fixed points on group structure, the rigid structure of uniform groups,  arithmetic properties of the arithmetic fundamental groups $\GG_S^T$, existence of Minkowski units, etc. In this subsection, we will give an outline of how these ingredients are combined together.

\medskip

$\bullet$ \emph{Criteria for infinitude of $\GG_S^T.$} To simplify, let us consider the context of Theorem \ref{maintheo1bis}. In the statement of that theorem, we refer to the need for $T$ to be "large enough" and here we wish to explain this a bit more.  In order to arrange to have enough fixed points,  we want  to take \begin{eqnarray}\label{inegalite000} |T|\geq \alpha s  + \beta, \end{eqnarray} with $\alpha$ and $\beta$ depending  on $\K$. 
On the other hand, by the theorem of Golod-Shafarevich,  the group $\GG_S^T$ is infinite when the $p$-rank of $\GG_S^T$ is sufficiently large. To be more exact, if \begin{eqnarray} \label{inegalite001} d_p \GG_S^T \geq 2+2\sqrt{|T|+r_1+r_2+1}, \end{eqnarray} where $(r_1,r_2)$ is the signature of $\K$, the pro-$p$ group $\GG_S^T$ is infinite (see for example \cite{Maire-JTNB}). Moreover, the $p$-rank of $\GG_S^T$ is at least  $s$ (because of the choice of $S$ and $T$).
 Hence, by (\ref{inegalite000}) and (\ref{inegalite001}),
 one can guarantee the infiniteness of $\GG_S^T$ by taking $s$ sufficiently large, \emph{i.e.} by introducing sufficiently many fixed points.   

\

$\bullet$ \emph{Uniform groups} (Part \ref{PartI}). Next we turn to the situation where a cyclic group $\langle \sigma \rangle$ of order $\ell$, with $\ell ~|~ p-1$, acts on a uniform group $\Gamma$. In particular we focus on the subgroup $\gamchapzero$ generated by the fixed points and its normal closure in $\Gamma$, denoted $\gamchap$. 
Here, the key result is Proposition \ref{pointfixe-uniform}:  it specifies generators for $\gamchap$ and is crucial for the rest of our work. When $\Gamma$ is FAb, the quotient group $G:=\Gamma/\gamchap$,  is a finite $p$-group. Moreover, when $\sigma$ is of order $2$, $G$ is abelian!

\medskip

$\bullet$ \emph{The choice of the prime ideals} (Part \ref{PartII}.) We fix $\K$ and consider varying sets $S,T$ where $\L \subseteq \K_S^T$ has Galois group $\Gamma=\Gal(\L/\K)$.  To simplify the exposition, we now assume $\sigma$ has order $2$.  Since $G$ is abelian, the field $\F$ fixed by $\gamchap$ is abelian over $\K$. If morevoer $\Gamma_\sigma$ is supported at $S$ then $\F$ is contained in the $p$-Hilbert class field $\K^H$.  To simplify further, let us assume $\F=\K^H$. The choice of prime ideals $\p$ of $S$ is based  on the following desired outcomes: (i) to create enough fixed points for the action of $\sigma$; (ii) to control the generators of $\GG_S(\F)$ via their inertia groups. Typically, the group $G$, acts trivially on the new ramification in $G_S^{{ab}}(\F)$.

\

To show the existence of such prime ideals, one  uses Kummer theory and  the Chebotarev density Theorem.  In order to do this, we require information about the units of the number field $\F$, namely we need $\F$ to contain "Minkowski units". {To be more precise}, let $\G=\Gal(\F/\k)$; we say that $\F$ has a Minkowski unit if the  quotient   $\O_{\F}^\times/(\O_{\F}^\times)^p$ contains a non-trivial  $\fq_p[\G]$-free module.  Note that we are not  in the semisimple case as $p\mid |\G|$! This delicate and interesting question has been studied in recent work of Ozaki \cite{Ozaki}: to estimate the  rank of the maximal free $\fq_p[\G]$-module of  $\O_{\F}^\times/(\O_{\F}^\times)^p$. Our idea here is to introduce a set $T$ and control the  $\fq_p[\G]$-structure  of  $T$-units of $\F$.   
\medskip

$\bullet$ \emph{The strategy} (\ref{PartIII}).
There exists a morphism of $\fq_p[G]$-modules $$\psi : \GG_S^{ab}(\F)/p \twoheadrightarrow \gamchap/\Phi(\gamchap).$$ 
The map $\psi$ dictates the compatibility of two  $\fq_p[G]$-modules, one of which comes from arithmetic considerations, and the other from group-theoretical ones. We suspect that the exploitation of this kind of compatibility can be useful in many other contexts.

\medskip

We now give some examples for which the   structures  of $\GG_S^{ab}(\F)/p$ and of $\gamchap/\Phi(\gamchap)$ as  $\fq_p[G]$-modules are not  compatible. 
Typically, the given situations are those for which the morphism $\psi$ is deduced from a  $\fq_p[G]$-module $\MM$ on which  $G$ acts trivially, namely we have a diagram as follows: 
$$\xymatrix{& \MM=(\fq_p)^{\oplus^s} \ar@{^(->}[ld]\ar@{.>>}[rd]^\psi& \\
\GG_S^{ab}(\F)/p \ar@{->>}[rr]  & &\gamchap/\Phi(\gamchap)
}$$ 

From the above diagram, one obtains a contradiction since  $G=\Gamma/\gamchap$ does not act trivially on  $\gamchap/\Phi(\gamchap)$.  This explains the relevance of  the notion of the action of $\sigma$ being "\fpm" that was introduced in Definition \ref{defi-fpm}.



\part{Uniform  groups and Fixed Points} \label{PartI}

Let $p$ be a prime number and let  $\Gamma$ be a finitely generated pro-$p$ group. 

\begin{itemize}
\item For two elements
$x,y$
of
$\Gamma$,  denote by $x^y:=y^{-1}xy$ the conjugate of $x$ by $y$, and by $[x,y]=x^{-1}x^y$ the commutator of $x$ and $y$. Put $[\Gamma, \Gamma]=\langle [x,y], x, y \in \Gamma\rangle$ and $\Phi(\Gamma)=[\Gamma,\Gamma]\Gamma^p$;
\item Let $\Gamma^{ab}:=\Gamma/[\Gamma,\Gamma]$ be the maximal abelian quotient of  $\Gamma$;
\item  The Frattini quotient  $\Gamma^{\pel}:=\Gamma/\Phi(\Gamma)$  is the maximal abelian $p$-elementary quotient of $\Gamma$;
\item Denote by   $d_p(\Gamma)=\dim_{\fq_p} H^1(\Gamma,\fq_p)=\dim_{\fq_p}\Gamma^{\pel}$ the $p$-rank of $\Gamma$: by the Burnside Basis Theorem, it is the minimal number of generators of $\Gamma$.
\end{itemize}

  \section{Schur-Zassenhaus}
  For this paragraph our reference is the book of  Ribes and Zalesskii \cite[Chapter 4]{RZ}.

  If $\Gamma$ is a finitely generated  pro-$p$ group of $p$-rank $d$, denote by $\Aut(\Gamma)$ the group of automorphisms  (always continous) of $\Gamma$.
  Recall that the kernel of the morphism $\ker\big(\Aut(\Gamma) \rightarrow \Aut(\Gamma^{\pel})\big)$is a pro-$p$ group  and that $\Aut(\Gamma^{\pel}) \simeq \Gl_d(\fq_p)$.
   Let us start with the following well-known result  which is crucial in our context: 
 \begin{theo}[Schur-Zassenhaus]\label{ZST}
Let  $1\longrightarrow \Gamma \longrightarrow \G \longrightarrow \G/\Gamma \longrightarrow 1$ be an exact sequence of profinite groups, where $\Gamma$ is a  pro-$p$ group finitely generated and where $\G/\Gamma$ is finite of order coprime to $p$.
Then the  group $\G$ has   a subgroup $\Delta_0$ isomorphic to the quotient $\Delta=\G/\Gamma$ and $\Delta_0$ is unique up to conjugation in $\G$. In particular:
$\G = \Gamma \rtimes \Delta_0 \simeq \Gamma \rtimes \Delta$. In other words, the  pointed set $H^1(\Delta,  \Gamma )$ is reduced to  $\{[0]\}$.

\end{theo}

\begin{proof}
See for example Theorem 2.3.15, \cite{RZ}.
\end{proof}

Let us now consider a  finitely generated  pro-$p$ group $\Gamma$ equipped with an automorphism $\sigma \in \Aut(\Gamma)$ of order coprime to $p$.
To simplify, we moreover assume here that the order of $\sigma $ is a prime number $\ell$.
 
 \begin{defi}
 Denote by  $$\gamchapzero:=\langle \gamma \in \Gamma, \ \sigma(\gamma)=\gamma \rangle$$
the closed subgroup generated by the fixed point of $\Gamma$
and by 
$$\gamchap:=\gamchapzero^{\Norm},$$ 
 the normal closure of $\gamchapzero$ in $\Gamma$. 
\end{defi}

 Of course, $\sigma$ acts trivially on $\gamchapzero$  and  $ \sigma \in\Aut(\gamchap)$.

\begin{defi} 
We say that the action of $\sigma$ on $\Gamma$ is  \emph{Fixed-Point-Free (FPF)} if  $\gamchapzero=\{e\}$.
\end{defi}

    Recall first a well-known result that shows the rigidity of the FPF-notion.

 \begin{prop}\label{FPF}
Let $\Gamma$ be a pro-$p$ group and let $\sigma \in \Aut(\Gamma)$ of order coprime to $p$.
If the action of $\sigma$ on $\Gamma$ is FPF, then $\Gamma$ is nilpotent. Moreover if $\sigma$ is of order $\ell=2$, then $\Gamma$ is abelian and if $\sigma$ is of order $3$ then $\Gamma$ is nilpotent of class at most~$2$. For $\ell \geq 5$, the nilpotency class of $\Gamma$ is at most  $\displaystyle{n(\ell):= \frac{(\ell-1)^{2^{\ell-1}-1}-1}{\ell-2}}$.
\end{prop} 
    
\begin{proof}
See Corollary 4.6.10, \cite{RZ}.
\end{proof}

   Now we may present the first step of our work.
    
    \begin{prop}\label{fini1}
    Let  $\Gamma$ be a finitely generated pro-$p$ group and $\sigma \in \Aut(\Gamma)$ of order $\ell$ coprime to $p$. Put $G:=\Gamma/\gamchap$.  Then the action of $\sigma$ on $G$ is FPF, so $G$ is nilpotent. If moreover  $\Gamma$ is {FAb} then  $G$ is a finite group.
    \end{prop}
    
    \begin{proof}
Consider the non-abelian Galois $\langle \sigma \rangle$-cohomology of the sequence:
$$1 \longrightarrow \gamchap \longrightarrow \Gamma \longrightarrow G \longrightarrow 1,$$
to obtain the sequence of pointed sets: $$0 \longrightarrow H^0(\langle \sigma \rangle,  \gamchap ) \longrightarrow H^0(\langle \sigma \rangle,  \Gamma) \longrightarrow H^0(\langle \sigma \rangle,  G ) \longrightarrow H^1(\langle \sigma \rangle,  \gamchap ) \longrightarrow \cdots$$
By the Schur-Zassenhaus Theorem \ref{ZST}, $H^1(\langle \sigma \rangle,  \gamchap )=\{[0]\}$ and then as $\gamchapzero=H^0(\langle \sigma \rangle,  \Gamma )=H^0(\langle \sigma \rangle,  \gamchap)$, one obtains $H^0(\langle \sigma \rangle,  G)=\{[0]\}$: in other words, the action of $\sigma$ on $G$ is FPF. Then by Proposition \ref{FPF} the pro-$p$ group $G$ is nilpotent of class at most $n(\ell)$. Moreover if $\Gamma$ is {FAb}, the pro-$p$ group $G$ is also {FAb}, one concludes that $G$ is finite.
    \end{proof}

\medskip

\section{Uniform pro-$p$ groups}

We first recall some basic facts about $p$-adic analytic groups (Lie groups over $\Q_p$). The main references for this section are \cite{DSMN} and \cite{lazard}.

Let $\Gamma$ be a $p$-adic  analytic pro-$p$ group: the profinite group $\Gamma$ is a closed subgroup of $\Gl_m(\Z_p)$ for some integer $m$. The group $\Gamma$ is called \emph{powerful} if $ [\Gamma,\Gamma] \subset \Gamma^p$ ($ [\Gamma,\Gamma] \subset \Gamma^4$ when $p=2$); a powerful pro-$p$ group $\Gamma$ is said  \emph{uniform} if it is torsion free. Let $\dim(\Gamma)$ be the dimension of  $\Gamma$ as analytic variety.

\begin{theo}
Every  $p$-adic analytic pro-$p$ group contains an open uniform subgroup. 
\end{theo}

For $i \geq 1$, denote by  $\Gamma_{i+1}=\Gamma_i^p[\Gamma,\Gamma_i]$ where $\Gamma_1=\Gamma$: it is the  $p$-central descending series of the   pro-$p$ group $\Gamma$. 

\begin{theo}
A powerful pro-$p$ group $\Gamma $ is uniform if and only if for $i\geq 1$,
the map $x\mapsto x^p$ induces an isomorphism between $\Gamma_{i}/\Gamma_{i+1}$ and $ \Gamma/\Gamma_2$.
\end{theo}

\begin{rema} 
It is conjectured that a torsion-free $p$-adic analytic pro-$p$ group satisfying
   $d_p(\Gamma)= \dim(\Gamma)$ must be uniform.
This is known to be the case when  $ \dim(\Gamma) < p$, see  \cite{Klopsch-Snopce}.
\end{rema}

\subsection{Additive law and automorphisms} \label{section-additive}
   
Let $\Gamma$ be a uniform pro-$p$ group.
If $\{x_1, \cdots, x_d\}$ is a minimal system of (topological) generators of $\Gamma$, then
  $\{x_1\Phi(\Gamma),\cdots, x_d\Phi(\Gamma\}$ forms a basis of $\Gamma^{\pel}=\Gamma/\Phi(\Gamma)$.
The group $\Gamma $ being uniform, the morphism $x\mapsto x^{p^n}$ induces an isomorphism $\psi_n$ between $\Gamma$ and $\Gamma_{n+1}$. By taking the limit on the  $p^n$th roots,  the group $\Gamma$ can be equipped with an additive law {(and we denote by $\Gamma_+$ this "new" group)}. More precisely, put $x+_n y = \psi_{n}^{-1}(x^{p^n}y^{p^n})$ and  $$x+y:= \lim_{n \rightarrow \infty} (x+_n y).$$ 
Then  $\Gamma_{+} := \Z_p x_1 \oplus \cdots \oplus  \Z_p x_d$ is a group isomorphic to $\Z_p^d$.

\begin{theo}[\cite{DSMN}, Theorem 4.9] \label{unicite}
Let $\Gamma$ be a uniform pro-$p$ group. Then $\Gamma=\overline{\langle x_1 \rangle}\cdots \overline{\langle x_d\rangle}$. In other words, for every $x\in \Gamma$, there exists a unique  $d$-tuple $(a_1,\cdots, a_d)\in \Z_p^d$ such that $x=x_1^{a_1}\cdots x_n^{a_n}$.
Moreover the map
$$\begin{array}{rcl}
\varphi : \Gamma & \longrightarrow & \Gamma_{+} \\
x= x_1^{a_1} \cdots x_d^{a_d} & \mapsto & a_1 x_1 \oplus \cdots \oplus a_n x_n
\end{array}$$
is a homeomorphism (not necessarily of groups).
\end{theo}

Let us fix $\sigma \in \Aut(\Gamma)$. It is not difficult to see that $\sigma(x)+_n \sigma(y)=\sigma(x+_n y)$. Hence, by passing to the limit, the action of  $\sigma$  
becomes a linear action on $\Gamma_{+}$, \emph{i.e.} $\sigma \in \Gl_d(\Z_p)$ (see \S 4.3 of \cite{DSMN}). One needs more to determine the Galois structure. 

\begin{theo}
The map  $\varphi$ induces an isomorphism of  $\fq_p[\langle \sigma \rangle]$-modules between $\Gamma^{\pel}$ and $\Gamma_{+}/p$.
\end{theo}

\begin{proof}
It suffices to note that $\varphi$ induces an isomorphism of groups between  $\Gamma^{\pel}$ and $\Gamma_{+}/p$: it is exactly Corollary 4.15 of \cite{DSMN}.
\end{proof}

\subsection{Semisimple action and fixed points}\label{semisimpleaction}

Recall the assumption that   $\sigma \in \Aut(\Gamma)$  is of finite order $\ell$, a prime number different from $p$.
We first recall a result which is valid in a more general context (not only for uniform groups), for  $\ell=2$, see \cite{HR} ; for  larger $\ell$ coprime to $p$, see \cite{Boston},  \cite{Wingberg2} and \cite{greenberg}. 
 
 \medskip

The action $\sigma$ on $\Gamma_{+}$ is semisimple, and the $\Z_p[\langle\sigma \rangle]$-module $\Gamma_{+}$ is projective. Hence the action of $\sigma$ on  $\Gamma_{+}/p$ lifts  uniquely  (up to isomorphism) to $\Gamma_{+}$ and then, one can find a family of  generators of  $\Gamma$ respecting this action, or that respects the  decomposition of $\Gamma_{+}$ as projective modules. 
If the action  of  $\sigma$ on the generators   $x_1,\cdots, x_d$ of $\Gamma_{+}$  has  $(a_{i,j})_{i,j}$ for matrix with  coefficients in $\Z_p$, we get $\displaystyle{\sigma(x_i)=\sum_{j=1}^d a_{i,j} x_j} \in \Gamma_{+}$, which becomes in  $\Gamma$: $\displaystyle{\sigma(x_i)=\prod_{j=1}^d x_j^{a_{i,j}}}$. 
 
 \medskip 
 Put $$r= \dim_{\fq_p} (\Gamma^{\pel})_\sigma \cdot$$ The integer $r$ corresponds to the dimension of the $\fq_p$-vector subspace of $\Gamma^{\pel}$ consisting of fixed points of $\Gamma^{\pel}$. The integer $r$ is the number of times that the trivial character appears  in the decomposition of the  $\fq_p[\langle\sigma\rangle]$-module $\Gamma^{\pel}$.

 \medskip
 
Now let us fix    a basis $\{x_1,\cdots, x_d\}$ of $\Gamma$ respecting  the decomposition  into irreducible characters following the action of  $\Gamma$. Suppose moreover that the  set  $\{x_1\cdots, x_r\}$ corresponds   to a basis of $(\Gamma^{\pel})_\sigma$. In particular,  $\sigma(x_i)=x_i$
 for $i=1,\cdots, r$. Clearly $\displaystyle{ \overline{\langle x_1\rangle}\cdots \overline{\langle x_r \rangle} \subseteq \gamchapzero}$.
For the reverse inclusion, one supposes moreover that  $\ell$ divides $p-1$.

 \medskip

{In the rest of this section, we will rely heavily on the following result.}

\begin{prop}\label{pointfixe-uniform} Let $\Gamma$ be a uniform  pro-$p$ group  and let $\sigma \in \Aut(\Gamma)$ be of order  $\ell$. Suppose that $\ell~|~(p-1)$. Then, with the notation introduced above, we have  $$\gamchapzero = \overline{\langle x_1\rangle}\cdots \overline{\langle x_r \rangle}=\langle x_1,\cdots, x_r\rangle \cdot$$
 \end{prop} 

\begin{proof} As  $\ell$ divides $p-1$, the  $\Q_p$-irreducible characters of $\langle \sigma \rangle$ are all of   degree $1$. In particular, by the choice of the $x_i$, we get that for $i>r$, $\sigma(x_i)=x_i^{\lambda_i}$, where $\lambda_i \in \Z_p \backslash \{1\}$. 

 Take $x \in \gamchapzero$ and let us write $x=x_1^{a_1}\cdots x_d^{a_d}$.
  Then
 $x=\sigma(x)$, if and only if,
$$\prod_{i=1}^d x_i^{a_i} = \prod_{i =1}^d \sigma(x_i)^{a_i} \cdot$$ As for $i=1,\cdots, r$, one has  $\sigma(x_i)=x_i$, we get
$$ \prod_{i>r}^d x_i^{a_i} = \prod_{i>r} x_i^{\lambda_i a_i}\cdot$$
Thanks to the uniqueness of the product in Theorem  \ref{unicite}, one deduces that for $i>r$,   $\lambda_i a_i= a_i$, \emph{i.e.},  $a_i=0$ because $\lambda_i \neq 1$.
One has proven that $\displaystyle{\gamchapzero = \overline{\langle x_1\rangle}\cdots \overline{\langle x_r \rangle}}$. On the other hand, trivially $\langle x_1,\cdots, x_r\rangle\subset  \gamchapzero$  and $\overline{\langle x_1\rangle}\cdots \overline{\langle x_r \rangle} \subset \langle x_1,\cdots, x_r\rangle$, which prove the desired equalities.
\end{proof}

\begin{rema}
In the above considerations, the uniform property of  $\Gamma$ is essential. By Proposition \ref{pointfixe-uniform} and when $ \ell~|~(p-1)$, the group $\gamchap$ is the normal subgroup of  $\Gamma$ generated by the conjugates of the $x_1,\cdots, x_r$. 
\end{rema}

\begin{rema}
The condition $\ell~|~(p-1)$ implies  that the irreducible characters of $\sigma$ are of degree $1$ and then $\varphi \circ \sigma = \sigma \circ \varphi$.
The existence of fixed points of $\sigma$ can be detected in   $\Gamma$ or in  $\Gamma_{+}$.
Let us remark that we can omit the condition "$\ell~|~(p-1)$" when  for all irreducible representations of  $\Z_p$-basis $\{ x_{i_1}, \cdots, x_{i_t}\}$, one has $x_{i_j}x_{i_k}=x_{i_k}x_{i_j}$. If the group $\Gamma$ is obtained by a certain exponential of a  $\Z_p$-Lie algebra $\Ell$ (see  \S \ref{section-lie}), the condition on the commutativity can be tested  in $\Ell$: this remark should open some other perspectives.
\end{rema}

We recover here, with a weaker hypothesis, \emph{i.e.} $\ell~|~(p-1)$, the following corollary   used in \cite{Hajir-Maire}:

\begin{coro}\label{coro-pointfixe} Let $\Gamma$ be a uniform pro-$p$ group.
Under previous conditions,  $\displaystyle{\gamchapzero=\{e\}}$ if and only if $(\Gamma^{\pel})_\sigma=\{\overline{e}\}$.
\end{coro}

  \begin{coro} \label{coro-generateurs}
  {Under the conditions of Proposition \ref{pointfixe-uniform}, we have
   $d_p (\Gamma/\gamchap) = d-r$ where $d=d_p\Gamma$ and $r=\dim_{\fq_p} (\Gamma^{\pel})_\sigma$.}
  \end{coro}
  
  \begin{proof}
  Put $G= \Gamma/\gamchap$. Consider the minimal system of generators $(x_i)_{i=1,\cdots,d}$ of $\Gamma$ introduced above, satisfying in particular that 
 $\sigma(x_i)=x_i$ for $i=1,\cdots, r$.  The group $\gamchap $ contains the elements $x_1,\cdots, x_r$. The  quotient $G$ is  topologically generated by the classes $x_i\gamchap$, $i>r$, so $d_p G \leq d-r$.
 In fact, the classes $(x_i\gamchap)_{i>r}$ form a minimal system  of generators of $G$: indeed, if not it would show that (possibly after renumbering) the   class $x_{r+1}\gamchap$ can be expressed  in terms of the classes $x_i\gamchap$, $i\geq r+2$, which would imply that the class $x_{r+1}\Phi(\Gamma)$ could be written in terms of the classes $\big(x_j\Phi(\Gamma)\big)_{j\neq r+1}$, which contradicts  the minimality of $\{x_1,\cdots, x_d\}$. Hence $d_pG =d-r$.
  \end{proof}
  
\subsection{On the group $\gamchap$}

 Let us conserve the notations and assumptions of the preceding subsection; in particular $\Gamma$ is uniform,  $\sigma \in \Aut(\Gamma)$ is of prime order $\ell$ and $\ell~|~(p-1)$. Recall that  $\displaystyle{\gamchapzero = \langle x_1 \cdots  x_r \rangle}$ and put $G=\Gamma/\gamchap$. 

\medskip

By Proposition \ref{fini1}, if $\Gamma$ is {FAb}, the group $\gamchap$ is open in $\Gamma$, the quotient $G=\Gamma/\gamchap$ is finite and $\Z_p[[G]] \simeq \Z_p[G]$.

\begin{rema}\label{longueur-derivee}
If $\Gamma$ is {FAb}, the group $G=\Gamma/\gamchap$ is a finite  $p$-group of $p$-rank at most   $d$ and having an automorphism  $\sigma$ acting without non-trivial fixed points. 
By Shalev \cite{Shalev}, it is possible to give 
 an upper bound for the solvability length ${\rm dl}(G)$ of $G$ which depends only on  $d$: ${\rm dl}(G) \leq 2^{d+1}-d-4+\lceil \log_2 d\rceil$. 
By taking the proof  of Shalev (lemma 4.4, lemma 4.5 and Proposition 4.6 of \cite{Shalev}), we remark 
that a key point is the number of distinct eigenvalues of $\sigma$. We note that in \cite{Shalev}, a slightly different notion of $G$ being \og uniform \fg \ is used: in Shalev's terminology, for  a such group $G$, one has ${\rm dl}(G) \leq 2^{\ell-1}-1$.  If we remove the condition of $G$ being uniform in the sense of Shalev, we have a weaker bound ${\rm dl}(G) \leq d(2^{\ell-1}-1) + \lceil \log_2( d) \rceil$. One should compare these bounds to the bounds  $m(\ell)$ coming from  \cite{RZ}, Corollary 4.6.10 (see Proposition \ref{FPF}) and given in the introduction (by recalling  that ${\rm dl}(G) \leq \log_2(n(\ell) +1)$).
\end{rema}

\medskip

 Recall now as  $G$ is a pro-$p$ group,  the ring $\Z_p[[G]]$ (resp. $\fq_p[[G]]$) is a local ring, with maximal ideal the augmentation ideal $\ker(\Z_p[[G]] \rightarrow \fq_p)$ (resp. $\ker(\fq_p[[G]] \rightarrow \fq_p)$). 
 The ring  $\Z_p[[G]]$ (resp. $\fq_p[[G]]$) acts  by conjugation on $\gamchap/[\gamchap,\gamchap]$ (resp. $\gamchappel$). 
 The following proposition gives  a system of minimal  generators of this action.
 
\begin{prop} \label{generateurs} 

\begin{enumerate}
\item[(i)]  The cosets $x_1 \Phi(\gamchap),\cdots, x_r\Phi(\gamchap)$  form a minimal system  of generators of the quotient $\gamchappel$   of $\gamchap$ seen as  $\fq_p[[G]]$-module. In particular $d_p \gamchap \geq r$.
\item[(ii)] The automorphism $\sigma$ acts trivially on $\displaystyle{\big(\Gamma_\sigma^{ab}\big)_G}$.
\end{enumerate}
\end{prop}

\begin{proof}
(i) By Proposition \ref{pointfixe-uniform}, the group $\displaystyle{\gamchapzero}$ is topologically generated by the elements  $x_1,\cdots, x_r$: they form a minimal system of generators.  Thus the quotient $\gamchappel$ is topologically generated by the $G$-conjuguates $G\cdot x_i\Phi(\gamchap)$ of the classes of the  $x_i$, $i=1,\cdots, r$, in $\gamchappel$. 
Consider now the exact sequence $$ \cdots H_2(\Gamma,\fq_p) \longrightarrow H_2(G,\fq_p) \longrightarrow \big(\gamchappel\big)_G \longrightarrow \Gamma^{\pel} \twoheadrightarrow G^{\pel},$$
 coming from the short exact sequence $1\longrightarrow \gamchap\longrightarrow \Gamma \longrightarrow G \longrightarrow 1$. 
The automorphism $\sigma$ acts ont these exact sequences. As the group $\gamchap$ contains the elements $x_1,\cdots, x_r$, the action of $\sigma$ on $G^{\pel}$ has no non-trivial fixed points. Moreover as  $\gamchappel$ is generated by the   $G$-conjuguates of the classes of the  $x_i$, $i=1,\cdots, r$, one gets that $\sigma$ acts trivially on $\big(\gamchappel\big)_G$. By comparing the character of the action of $\sigma$ on the initial exact sequence, one obtains that $d_p \big(\gamchappel\big)_G =r$. Thus by  Nakayama's lemma, the classes $x_1\Phi(\gamchap),\cdots, x_r\Phi(\gamchap)$ form a minimal system of generators of the $\fq_p[[G]]$-module $\gamchappel$. In conclusion we get  $d_p \gamchap \geq r$.

Now (ii)  is obvious: the group $\displaystyle{\Gamma_\sigma^{ab}}$ is generated by the  $G$-conjuguates of the  classes of the $x_i\Phi(\gamchap)$, $i=1,\cdots,r$, hence $\sigma$ acts trivially on   
$\displaystyle{\big(\Gamma_\sigma^{ab}\big)_G}$.
\end{proof}

\begin{rema} As consequence of the proof, we get an exact sequence:
$$1\longrightarrow \big(\gamchappel\big)_G \longrightarrow \Gamma^{\pel} \longrightarrow G^{\pel} \longrightarrow 1,$$
and then  $H_2(\Gamma,\fq_p) \twoheadrightarrow H_2(G,\fq_p)$.
\end{rema}

Recall that for a uniform group $\Gamma$ of dimension $d$, for all closed subgroup $\U$ of $\Gamma$, one has  $d_p \U \leq d$.

\begin{rema} Assume $G$ finite. On can define its maximal free sub-$\fq_p[G]$-module $\big(\gamchappel\big)_0$ of  $\gamchappel$.
As the $p$-rank of $\gamchap$ is smaller than $d$, one sees that   $\displaystyle{\big(\gamchappel\big)_0}$  is trivial as soon as $|G| >  d$.
\end{rema}

\medskip

We now recall a notion introduced in Definition \ref{defi-fpm}: the action of $\sigma$ on the group $\Gamma$  is called {\fpm} (FPMF) if $G=\Gamma/\Gamma_\sigma$ does not act trivially on $\gamchap/\Phi(\gamchap)=\gamchappel$.

\begin{prop}\label{critere-non-size} If the action of $\sigma$ on $\Gamma$ is not \fpm, then $\gamchap=\gamchapzero$ and $d_p \gamchap = r$. 
\end{prop}

\begin{proof}
It is a consequence of Proposition \ref{generateurs}.
\end{proof}

\begin{prop}\label{size}
Let $\Gamma$ be a uniform group of dimension dimension $d>1$. Suppose   $\sigma\in \Aut(\Gamma)$ of order   $\ell=2$. Recall that $r=\dim_{\fq_p} (\Gamma^{\pel})_{\sigma}$.
\begin{enumerate}
\item[(i)] If $t_\sigma(\Gamma)=(1,d-1)$ and if $\gamchap$ is open (which is the case if  $\Gamma$ is {FAb}), then the action of $\sigma$ on  $\Gamma$ is \fpm.
\item[(ii)] If $t_\sigma(\Gamma)=(d-1,1)$ and if   $d_p \gamchap >r$, then  $\gamchap$ is uniform.
\end{enumerate}
\end{prop}

\begin{proof}
(i)  If $d_p \gamchap=1$, the group $\gamchap$ is generated by only one element and then is  procyclic. If $\gamchap$ is open, the quotient $\Gamma/\gamchap$ is finite and then the $p$-adic analytic group  $\Gamma$ is of dimension $1$ which is a contradiction. Thus, $d_p \gamchap>1$ and the action of $\sigma$ on $\Gamma$ is \fpm \ thanks to Proposition \ref{critere-non-size}.

(ii) Here  $d_p \gamchap=d$ which is equivalent to $\gamchap$ being a uniform group (of dimension $d$).
\end{proof}

\begin{rema} We will see in Proposition \ref{calcul-} that when $\Gamma$ is FAb and uniform of dimension $d$, $t_\sigma(\Gamma) \neq (d-1,1)$.
\end{rema}

\begin{rema}[See \cite{LM}]
For every  uniform group $\Gamma$, there exists an open subgroup  $\Gamma_0$ such that for all open subgroup   $\U$ of $\Gamma_0$, one has $d_p \U \geq \dim(\Ell(\Gamma))$, where $\dim(\Ell(\Gamma))$ is the dimension of the Lie algebra associated to $\Gamma$ (see the next section).  
If  moreover $\Gamma$ is  {FAb}, one has also $\dim(\Ell(\Gamma)) < \dim(\Gamma)=d_p \Gamma$.
\end{rema}


\subsection{Uniform groups and Lie algebras}\label{section-lie}

\subsubsection{The correspondence} \label{section-correspondance}

Consider a uniform group $\Gamma $ of dimension $d$. We have seen how to associate to $\Gamma$ a uniform abelian group  $\Gamma_{+} \simeq \Z_p^d$. In fact, this group is naturally equipped with more algebraic structure, as we now explain.

\

For $x, y \in \Gamma$, put $(x,y)_n:=\psi_{2n}([x^{p^n},y^{p^n}])$ and define $$(x,y)=\lim_{n\rightarrow \infty}(x,y)_n\cdot$$

\medskip

\begin{theo}[\cite{DSMN},Theorem 4.30]
The $\Z_p$-module $\Gamma_{+}$ equipped with the bracket $( \cdot, \cdot)$ is a $\Z_p$-Lie algebra of dimension $d$. Denote by $\Ll_\Gamma$ this new Lie algebra.
\end{theo}

\begin{rema} Recall that each $\sigma \in \Aut(\Gamma)$ induces an automorphism of  $\Gamma_{+}$.
By noting that  $\sigma((x,y)_n)=(\sigma(x),\sigma(y))_n$, we see that  $\sigma(x,y)=(\sigma(x),\sigma(y))$, so $\sigma$ becomes an automorphism of  the $\Z_p$-Lie algebra $\Ll_{\Gamma}$.
\end{rema}

Remark that as $\Gamma$ is uniform, thus $[\Gamma,\Gamma] \subset \Gamma^{2p}$ and  $(x,y)_n \in \Gamma^{2p}$; by passing to the limit, one obtains: $(\Ll_{\Gamma},\Ll_{\Gamma})\subset 2p \  \Ll_{\Gamma}$.

\begin{defi}
A $\Z_p$-Lie algebra $\Ll$ is called  powerful if $(\Ll,\Ll) \subset 2p \Ll$. 
\end{defi}

Recall now the following correspondence.

\begin{theo}[\cite{DSMN}, Theorem 9.10] \label{Correspondance}
There exists a bijective  correspondence between the category of uniform groups of dimension $d$ and the  category of powerful $\Z_p$-Lie algebras of dimension $d$.
\end{theo}

Given a uniform group of dimension $d$, we have already seen how to associate to it  a $\Z_p$-Lie algebra of dimension $d$. 
The inverse map is obtained by using  the development  of Campbell-Hausdorff $\Phi$ (see section 9.4 of \cite{DSMN}), which we now sketch. Using the formal series of $\Q_p[[X]]$  $$E(X)=\sum_{n\geq 0} \frac{1}{n !}X^n \ \ {\rm and} \ \ L(X)=\sum_{n\geq 1} \frac{(-1)^{n+1}}{n} X^n,$$
we define $$\Phi(X,Y)=L(E(X)E(Y)-1)\cdot $$

We consider the series $\Phi$  in the Lie algebra $\Q_p((X,Y))$ equipped with the bracket  $[X,Y]=XY-YX$. 
For $U_1,\cdots, U_t \in \Q_p((X,Y))$, let us define by induction  $[U_1,\cdots, U_t]=[[U_1,\cdots, U_{t-1}],U_t]$. If $e$ is the multiset $e=(e_1,\cdots,e_s)$, $e_i\geq 1$ for all $i$,  put
$[X,Y]_e=[X,X(e_1),Y(e_2), \cdots ]$, where $X(e_i)$ (resp. $Y(e_j)$) denotes the constant $e_i$-tuple  $X(e_i)=(X,\cdots, X)$ (resp. the $e_j$-tuple $Y(e_j)=(Y,\cdots, Y)$).  Hence $[X,Y]_e$ is a bracket of length $\langle e \rangle := 1+ e_1+\cdots + e_s$. 
 
In $\Q_p((X,Y))$, let us write $\displaystyle{ \Phi = \sum_n u_n(X,Y)}$, where  $u_n$ denote the  elements of  degree  $n$ which can be expressed as a sum
$$u_n(X,Y)=\sum_{\langle e\rangle =n-1} q_e [X,Y]_e,$$  where $q_e \in \Q$.
One translate these formulas to the   $\Z_p$-Lie algebras $\Ll$ equipped with the bracket  $(\cdot, \cdot)$. Suppose $\Ll$ powerful.
Then for all $x,y \in \Ll$, the specialization of $\Phi$ at the bracket $(x,y)$ converges in $\Ll$ (see Corollary   6.38 of \cite{DSMN}). 
Concretely, in the formulas one has replaced  $[X,Y]$ by $(x ,y)$.

\begin{theo}[\cite{DSMN}, Theorem 9.8 and Theorem 9.10]\label{theo-correspondance} Let  $\Ll$ be a powerful $\Z_p$-Lie algebra of  dimension $d$. Let  $\{x_1,\cdots, x_d\}$ be a $\Z_p$-basis of $\Ll$. 
The law $x*y=\Phi(x,y)$ makes $\Ll$ into a uniform group  $\Gamma_\Ll$ of dimension $d$, topologically generated by $\{x_1,\cdots, x_d\}$. 
Moreover $\Ll_{\Gamma_\Ll} \simeq \Ll$ and $\Gamma_{\Ll_\Gamma} \simeq \Gamma$.
\end{theo}

\begin{rema}\label{matrices}
Let us examine carefully the case where $\Ll$ is a  powerful sub-Lie-algebra  of the Lie algebra $\M_n(\Q_p)$ of $p$-adic $n\times n$ matrices  equipped  with the bracket $(A,B)=AB-BA$. 
Consider the map   "exponential" $\exp$ and "logarithm" $\log$ of matrices well-defined in our context  (see \S 6.3 of \cite{DSMN}):
$\xymatrix{ \Ll \ar@/^1pc/[rr]^{\exp} && \ar@/^1pc/[ll]^{\log}\exp(\Ll)
}$.
Thus for $A,B \in \Ll$, we get $\exp(A) \exp(B)= \exp(\Phi(A,B))$, where $\Phi$ is the Campbell-Hausdorff series (see Proposition 6.27) and then $\exp(\Ll)$  is isomorphic to the uniform group $\Gamma_{\Ll}$ (see Corollary 6.25 of \cite{DSMN}). 
\end{rema}

{Since we are especially interested in uniform groups which are FAb, we give a characterization of such groups, which is probably well-known to  specialists.}

\begin{prop}\label{critere-fab}
A uniform  group $\Gamma$ is FAb if and only if   $$\Ll_\Gamma(\Q_p)=(\Ll_\Gamma(\Q_p),\Ll_\Gamma(\Q_p)),$$ where $\Ll(\Q_p)$ is the $\Q_p$-Lie algebra obtained from $\Ll$ by extending the scalars to $\Q_p$.
\end{prop}

\begin{proof}
For every open subgroup $H$ of the uniform group $\Gamma$,  $\Ll_H(\Q_p)=\Ll_\Gamma(\Q_p)$. Hence,
one has to prove that  $\Gamma^{ab}$ is finite if and only if, $\Ll_\Gamma(\Q_p)=(\Ll_\Gamma(\Q_p),\Ll_\Gamma(\Q_p))$. 
Suppose $\Gamma^{ab}$ infinite. There exists a closed and normal subgroup  $H$ of $\Gamma$ such that $\Gamma/H \simeq \Z_p$. By Proposition 4.31 of \cite{DSMN}, the subgroup $H$ is uniform, the  $\Z_p$-Lie algebra  $\Ll_H$ is an ideal of $\Ll_\Gamma$, and $\Ll_{\Gamma/H} \simeq \Ll_\Gamma/\Ll_H$. As $\Gamma/H $ is abelian, the Lie algebra  $\Ll_{\Gamma/H}$ is commtutative (corollary 7.16 of \cite{DSMN}). In fact,  $\Ll_{\Gamma/H}=\Z_p$. Then $\Ll_H$ contains $[\Ll_\Gamma,\Ll_\Gamma]$; thus $\Ll_\Gamma/(\Ll_\Gamma,\Ll_\Gamma) \twoheadrightarrow \Ll_\Gamma/\Ll_H \simeq \Z_p$ and  therefore  $(\Ll_\Gamma(\Q_p),\Ll_\Gamma(\Q_p)) \subsetneq \Ll_\Gamma(\Q_p)$.

\

In the other direction, suppose that $(\Ll_\Gamma(\Q_p),\Ll_\Gamma(\Q_p)) \subsetneq \Ll_\Gamma(\Q_p)$, or equivalently that the  $\Z_p$-rank of $\Ll_\Gamma/(\Ll_\Gamma,\Ll_\Gamma)$ is not trivial. Put $\Ll_1=\Ll_\Gamma/(\Ll_\Gamma,\Ll_\Gamma)$. As  $\Z_p$-modules, let us write $\Ll_1=\Ll_0 \oplus \Tor (\Ll_1)$. It is then easy to  see that $\Tor(\Ll_1)$ is an ideal of the  Lie algebra $\Ll_1$. Thus consider the quotient $\Ll_0:=\Ll_1/\Tor(\Ll_1)$: it is a non trivial, commutative and torsion-free $\Z_p$-Lie algebra! By the  correspondence of Theorem \ref{Correspondance}, the algebra $\Ll_0$ corresponds to a uniform abelian group $\Gamma_0$ (by Corollary 7.16 of \cite{DSMN}). In fact, as $\Ll_0 \simeq \Z_p^t$, with $t>0$, on has $\Gamma_0\simeq  \Z_p^t$.
The algebra   $\Ll_0$ is also the quotient of $\Ll_\Gamma$ by the ideal  $\Ll_2$ generated by  $(\Ll_\Gamma,\Ll_\Gamma)$ and the lifts of $\Tor(\Ll_1)$.  By Proposition 7.15 of \cite{DSMN},  under the correspondence of  Theorem  \ref{Correspondance}, the algebra $\Ll_2$ corresponds to a uniform closed subgroup $H$ of $\Gamma$; moreover $\Gamma/H$ is uniform. Therefore as for the  previous implication, we get  $\Ll_{\Gamma/H} \simeq \Ll_\Gamma/\Ll_2 \simeq \Ll_0 \simeq \Z_p^t$ and then $\Gamma/H \simeq \Gamma_0 \simeq \Z_p^t$.
\end{proof}

The proof has shown  the following result:
\begin{coro}
A uniform group $\Gamma$ is {FAb} if and only if $\Gamma^{ab}$ is finite.
\end{coro}

In this  subsection, we have seen the relevance of  powerful $\Z_p$-Lie algebras and their automorphisms in our study. If moreover we restrict attention to FAb and uniform groups, one sees the importance of simple algebras. Indeed, it follows from definitions that  every $\Z_p$-Lie algebra  $\Ll$  which is simple or even semisimple (after extending scalars) produces a uniform FAb group by  Proposition \ref{critere-fab}.

\subsubsection{Lie algebras and  fixed points}

We now further explore the Lie algebra  $\Ll$ over  $\Q_p$. Denote by $(\cdot, \cdot)$ the Lie bracket of $\Ll$.
For algebras of dimension   $2$ or $3$, see for example \cite{Jacobson2}, \S I.4.

\begin{defi}
Let $\Ll$ be a Lie algebra and let  $\sigma \in \Aut(\Ll)$. Put  $\Ll_\sigma=\{ x \in \Ll, \sigma(x)=x\}$. 
\end{defi}

Let us introduce the notion of FAb algebra.

\begin{defi}
A Lie algebra $\Ll$ over $\Q_p$ is called {FAb} if $(\Ll,\Ll)=\Ll$. In particular a semisimple Lie algebra is FAb.
\end{defi}

  As for  pro-$p$ groups having a FPF automorphism $\sigma$ of order $\ell \neq p$,  the same phenomenon occurs  for Lie algebras.
  Indeed as a consequence of a result of  Borel and  Serre \cite{Borel-Serre} (cf the remark  of Jacobson \cite{Jacobson1}, page 281), we have the following Proposition.

\begin{prop}\label{lie-nilpotente}
Let $\Ll$ be a \emph{FAb} Lie algebra and let $\sigma \in \Aut(\Ll)$ of order $\ell$. Then $\Ll_\sigma \neq \{0\}$.
\end{prop}

\begin{proof}
Indeed, by Proposition 4 of \cite{Borel-Serre}, if $\Ll_\sigma=\{0\}$ then  $\Ll$ is nilpotent and the conclusion is obvious.
\end{proof}

An automorphism of order $\ell$ of a FAb $\Q_p$-Lie algebra must have a non-trivial fixed point. One finds again  Proposition \ref{FPF} in the context of uniform groups.
  If  $\sigma \in \Aut(\Ll)$ is of order $2$, as for pro-$p$ groups, one define the $\sigma$-type of $\Ll$  as  $t_\sigma(\Ll)=(a,b)$, where $a=\dim \ker(\sigma - \iota)$ and  $b=\dim \ker(\sigma + \iota)$, $\iota$ being the trivial automorphism. We have $a=\dim \Ll_\sigma$ and $b=d-a$ where $d=\dim \Ll$.

\begin{prop} \label{calcul-}
Let $\Ll$ be a FAb $\Q_p$-Lie algebra of   dimension $d$ and let $\sigma \in \Aut(\Ll)$ be of order $2$. Let  $t_\sigma(\Ll)=(a,b)$ be the $\sigma$-type of $\Ll$. Then  $a\neq 0$ and $b > 1$.
\end{prop}

\begin{proof} By Proposition \ref{lie-nilpotente}, the type $(0,d)$ is excluded.
Suppose  $\Ll$ of type $(d-1,1)$.
Take a $\Q_p$-basis $\{e_1,e_2,\cdots, e_{d-1},\varepsilon\}$ of $\Ll$   respecting the action $\sigma$, \emph{i.e.} for $i=1,\cdots, d-1$, 
$\sigma(e_i)=e_i$ and $\sigma(\varepsilon)=-\varepsilon$.
One then remarks that  $\sigma$ acts by $+1$  on $(e_i,e_j)$ and by $-1$ on $(e_i,\varepsilon)$: therefore $(e_i,e_j) \in \langle e_1 , \cdots, e_{d-1} \rangle$ and $(e_i,\varepsilon) \in \langle \varepsilon \rangle$.
Hence, for $i\neq j$, $\displaystyle{ (e_i,e_j)=\sum_{k=1}^{d-1} a_k(i,j) e_k}$, with $a_k \in \Q_p$, and also for $i=1,\cdots, d-1$, $(e_i,\varepsilon)=\lambda_i \varepsilon$.
As the Lie algebra is {FAb}, the matrix  $(a_k(i,j))_{((i,j),k)}$ of size $\displaystyle{\frac{(d-1)(d-2)}{2} \times (d-1)} $ must be of maximal rank, \emph{i.e.} $d-1$.
Also the vector  $(\lambda_1,\cdots, \lambda_{d-1})$ are non zero. 

Now the elements $(e_i)_i$ and $\varepsilon$ should verify the Jacobi identity; in particular one should have for $i\neq j$~: 
$$(e_i,(e_j,\varepsilon))+(e_j,(\varepsilon,e_i))+(\varepsilon,(e_i,e_j))=0 \cdot$$
Thus one gets
$$
\begin{array}{rcl} (e_i,(e_j,\varepsilon))+(e_j,(\varepsilon,e_i))+(\varepsilon,(e_i,e_j)) & = &\displaystyle{\lambda_j(e_i,\varepsilon) - \lambda_i(e_j,\varepsilon)+ \sum_{k=1}^{d-1}a_k(i,j) (\varepsilon, e_k) } \\
&=& \displaystyle{\lambda_j\lambda_i \varepsilon - \lambda_i\lambda_j \varepsilon - \sum_{k=1}^{d-1}a_k(i,j) \lambda_k \varepsilon}
\end{array}$$
and then $$\sum_{k=1}^{d-1}a_k(i,j) \lambda_k=0 \cdot$$
If the matrix $(a_k(i,j))_{((i,j),k)}$ is of maximal rank, then  $\lambda_k=0$ for all   $k$ and $\Ll$ is not {FAb}.
 \end{proof}

Applying the correspondence of uniform groups/Lie algebras, this proposition allows us to obtain  the following corollary:

\begin{coro}\label{coro-calcul-}
Let $\Gamma$ be a FAb uniform group of dimension $d$ and let $\sigma \in \Aut(\Gamma)$ be of order $2$.
Then  $t_\sigma(\Gamma)=(d-k,k)$ with $k \geq 2$. Therefore for  a {FAb}  uniform group of dimension $3$ the 
type  of every automorphism $\sigma$ of order $2$ is constant and equal to  $t_\sigma(\Gamma)=(1,2)$. \end{coro}

 On other hand, look at Lie algebras  $\Ll$ having few  fixed points. 
Consider, say, a Lie algebra $\Ll$ of dimension $4$ such that $\Ll_\sigma$ is of dimension $1$.
Let  $\{e_1,e_2,e_3, \varepsilon\}$ be a $\Q_p$-basis  of $\Ll$ respecting the action if $\sigma$, \emph{i.e.} here
$\sigma(e_i)=-e_i$ and $\sigma(\varepsilon)=\varepsilon$. 
Then  $(e_i,e_j) \in \langle \varepsilon \rangle$ and $(e_i,\varepsilon) \in \langle e_1,\cdots, e_3 \rangle$.
A  linear algebra computation similar to those of Proposition \ref{calcul-} shows  that $\Ll$ can not be FAb: necessarily, $\Ll/( \Ll,\Ll) \twoheadrightarrow  \Q_p$. 
The same holds for the  dimension $5$.
In fact, it is a general and well-known phenomenon  for semisimple Lie algebras $\Ll$. Indeed dimension of $\Ll_\sigma$ grows with the dimension of $\Ll$ (see Theorem 10 and Theorem 8 of \cite{Jacobson1}).

\section{Examples}\label{section-examples}

\subsection{First examples}

We first give   examples of non FAb uniform groups having an automorphism $\sigma$ of order  $2$ with fixed points.   We assume that $p>2$.

\subsubsection{Direct product} Consider  $\Z_p \times \Z_p=\langle x\rangle \times \langle y \rangle $ with $\sigma(x)=x$ and $\sigma(y)=y^{-1}$; $t_\sigma(\Gamma)=(1,1)$.
Then $\displaystyle{\gamchap=\gamchapzero=\langle x\rangle }$ and $G = \Gamma /\gamchap =\langle y \rangle \simeq \Z_p$.
 Here $\Z_p[[G]] \simeq \Z_p[[T]]$ and the action of $y$ on $x$ is trivial. 
 
 \subsubsection{Semidirect product} For $a \in \Z_p^\times$, 
 Consider $$\Gamma=\Gamma(a)=\langle x,y \ | \ y^x=y^a \rangle$$
 which can be realized concretely as $\Gamma= \langle \xi, \eta \rangle \subset \Gl_2(\Z_p)$ where
 $$
 \xi=\left(\begin{array}{cc}a^{-1} & 0 \\ 0 &1 \end{array} \right), \ \eta =\left(\begin{array}{cc}1 & p \\ 0 &1 \end{array} \right).
 $$
 Without loss of generality, we may take $a=1+p^k$.
 Let $\sigma \in \Aut(\Gamma)$  of order $2$ such that $\sigma(x)=x$ and $\sigma(y)=y^{-1}$. 
 Here, $t_\sigma(\Gamma)=(1,1)$. 
 One then has $\gamchap=\langle x, y^{a-1}\rangle$ and  $G \simeq \Z/p^k\Z$, where $k$ is the  $p$-adic valuation of $a-1$: the subgroup $\gamchap$ is open. One has $[x,y]=y^{p^k}$, therefore $[x,y^n]=y^{np^k}$. Put $x_i=\sigma^i(x)$ and let us use the additive notation. For $2 \leq n \leq p^{k}-1$, one has the relation
 $x_n=(1-n)x_0+nx_1$.  Thus $\displaystyle{(\gamchap)^{\pel}=\fq_p[G] \cdot x=\langle x_0, x_1 \rangle}$ which is
 of  $p$-rank $2$.
 The action of $\sigma$ on  $\Gamma$ is \fpm \ (see also Proposition \ref{size}).
  
  \medskip
 If we consider the decomposition of $\displaystyle{(\gamchap)^{\pel}}$ as the sum of   indecomposable modules (see \cite{CR}, lemma 64.2), one necessarily obtains $(\gamchap)^{\pel}\simeq_G I^{p^k-2}$, where $I:=\ker\big( \fq_p[G] \rightarrow \fq_p \big)$ is the augmentation ideal of $\fq_p[G]$.

 \subsubsection{The Heisenberg group}
 Let $\Ll$ be the $\Z_p$-Lie algebra of dimension $3$ generated by   the matrices $\displaystyle{x=\left(\begin{array}{cc}p & 0 \\ 0 &0 \end{array} \right), \ y =\left(\begin{array}{cc}0 & 0 \\ 0 &p \end{array} \right), \   z=\left(\begin{array}{cc}0 & p \\ 0 &0 \end{array} \right) }$, with bracket $[A,B]=AB-BA$.
 Hence $[x,y]=0$, $[x,z]=pz$ and $[y,z]=-pz$. Moreover $\Ll$ is  powerful   but not {FAb} because $[\Ll(\Q_p),\Ll(\Q_p)]=\Q_p z$.
 Denote by  $\Gamma$ the  uniform group  generated by the exponentials of $x,y$ and $z$:
 $$X=\exp(x)=\left(\begin{array}{cc}e^p & 0\\ 0 &1 \end{array} \right), \ Y=\exp(y)=\left(\begin{array}{cc}1 & 0 \\ 0 &e^p \end{array} \right), \ 
Z=\exp(z)=\left(\begin{array}{cc}1 & p \\ 0 &1 \end{array} \right),$$
under the correspondence of Theorem \ref{theo-correspondance} (see remark \ref{matrices}). The $\Z_p$-rank of $\Gamma^{ab}$ is equal to  $2$ and $\Gamma^{ab} \simeq \Z_p^2 \times \Z/p\Z$.
Let $\sigma=\sigma_A \in \Aut(\Gamma)$ be the automorphism of order $2$ corresponding to conjugation by the matrix $A=\left(\begin{array}{cc}1 &0 \\
0& -1
\end{array}\right)$: $\sigma(M)=A^{-1}MA$. Then $\displaystyle{\gamchapzero= \langle X,Y \rangle \simeq \Z_p^2}$ and $G:=\Gamma/\gamchap=\langle Z \gamchap \rangle \simeq \Z/p\Z$.
As here  $G$ is finite and as $\displaystyle{\gamchapzero}$ is of analytic dimension  $2$, necessarily $\displaystyle{\gamchapzero \subsetneq   \gamchap}$ and then $G$ does not act trivially on $\displaystyle{(\gamchap)^{\pel}}$.

\subsection{The group $\Sl_2^1(\Z_p)$} \label{sl2} As before, we assume $p>2$.
Let us start with the $\Z_p$-Lie algebra $\sl_2$ of dimension $3$ generated by the matrices $$x=\left(\begin{array}{cc}0 & p \\ 0 &0 \end{array} \right), \ y =\left(\begin{array}{cc}0 & 0 \\ p &0 \end{array} \right), \   z=\left(\begin{array}{cc}p & 0 \\ 0 &-p \end{array} \right) \cdot$$ The algebra $\sl_2$ is the subalgebra of the trace zero matrices for which the reduction    modulo $p$ is trivial.
One has the relations  $[x,y]=p z$, $[x,z]=-2px$ and $[y,z]=2py$.  As $[\sl_2,\sl_2]\subset p\cdot \sl_2$, the algebra $\sl_2$ is FAb and  powerful. 
Put
$$X=\exp(x)=\left(\begin{array}{cc}1 & p \\ 0 &1 \end{array} \right), \ Y=\exp(y)=\left(\begin{array}{cc}1 & 0 \\ p &1 \end{array} \right), \ 
Z=\exp(z)=\left(\begin{array}{cc}e^p & 0 \\ 0 &e^{-p} \end{array} \right) \cdot $$
Let $\Sl_2^1(\Z_p)$  be the subgroup of $\Sl_2(\Z_p)$ generated by the  $X,Y, Z$; it is the kernel of the reduction morphism $\Sl_2(\Z_p)\rightarrow \Sl_2(\fq_p)$. The group $\Sl_2^1(\Z_p)$ is FAb, uniform and of dimension  $3$.

\begin{prop}\label{prop-sl2} For every involution $\sigma$ of the uniform pro-$p$ group $\Sl_2^1(\Z_p)$, the action of $\sigma$ is \fpm.
\end{prop}

\begin{proof}
 The uniform pro-$p$ group $\Sl_2^1(\Z_p)$ is FAb and of dimension     $3$: by Corollary    \ref{coro-calcul-}, every  automorphism $\sigma$ of order  $2$  is of type $(1,2)$. One concludes with  Proposition \ref{size}.\end{proof}

By using the correspondence of  Section \ref{section-correspondance}, one can say more.
Consider the homomorphisms  "exponential" and "logarithm" of matrices:
$\xymatrix{ \sl_2 \ar@/^1pc/[rr]^{\exp} && \ar@/^1pc/[ll]^{\log}\Sl_2^1(\Z_p)
}$.
 
Given $A\in \Gl_2(\Z_p)$ and  $B \in \sl_2$,  one has $\exp(\sigma_A(B))=\sigma_A(\exp B)$, where $\sigma=\sigma_A$ is the  conjugation by $A$,  providing the passage from $\Aut(\Sl_2^1(\Z_p))$ to $\Aut(\sl_2)$, and \emph{vice  versa}.

Let $\tau $ be an involution of $\Sl_2^1(\Z_p)$. The automorphism $\tau$ induces  an involution  on $\sl_2$ and then on $\sl_2(\Q_p)$. The key point comes from the fact that every  automorphism of  $\sl_2(\Q_p)$ is the automorphism  $\sigma_A$ of conjugation by  a certain matrix $A \in \Gl_2(\Q_p)$ (see \cite{Seligman}, Theorem 1); thus $\tau(T)=\sigma_A(T)=A^{-1}TA$. Now,  one can assume that the coefficients of  $A$ are in $\Z_p$. Then, as $\sigma$ is of order $2$  and acts on $\sl_2$,  the minimal polynomial of $\A$ is $X^2-1$ or $X^2-\varepsilon$, for  $\varepsilon \in \Z_p^\times \backslash (\Z_p^\times)^2$. 
And, at the end, one can assume that  $A\in \Gl_2(\Z_p)$.
Hence the matrix $A$ is similar in $\Gl_2(\Z_p) $ to  $D=\left(\begin{array}{cc}1 & 0 \\ 0 &-1 \end{array} \right)$
or to $D_\varepsilon= \left(\begin{array}{cc}0 & \varepsilon \\ 1 &0 \end{array} \right)$ (see \cite{Pomfret}).

Let $\sigma=\sigma_D$ and $\sigma_\varepsilon=\sigma_{D_\varepsilon}$ be the two involutions of  $\Sl_2^1(\Z_p)$ corresponding to the conjugation by $D$ or by $D_\epsilon$ respectively. They are  defined by  \begin{itemize}
\item  $\sigma(X)=X^{-1}$, $\sigma(Y)=Y^{-1}$ and $\sigma(Z)=Z$;
\item $\sigma_\varepsilon(X)=Y^{1/\varepsilon}$, $\sigma_\varepsilon(Y)=X^\varepsilon$ and $\sigma_\varepsilon(Z)=Z^{-1}$.
\end{itemize}

Hence there exists  $M \in \Gl_2(\Z_p)$ such that $MAM^{-1}=D$ or $D_\varepsilon$. Thus $\tau= \sigma_A= \sigma_M^{-1}\sigma \sigma_M$ or $\tau=\sigma_M^{-1}\sigma_\varepsilon \sigma_M$. As the matrices $M$ and $A$ are in $\Gl_2(\Z_p)$, $\sigma_{M}$ and $\sigma_A$ are in $ \Aut(\Sl_2^1(\Z_p))$:  the correspondence gives that the identity $\tau= \sigma_M^{-1}\sigma \sigma_M$ can be seen to be in  $\Aut(\Sl_2^1(\Z_p))$ showing that     $\tau$ is conjuguate to $\sigma$ or to $\sigma_\varepsilon$ in $\Aut(\Sl_2^1(\Z_p))$.  
Lastly,  let us remark that for $\sigma$, the situation is easy to describe: one has $\gamchapzero=\langle Z\rangle$. Let  $\gamchap=\langle Z \rangle^{\rm Norm}$ be the normal subgroup of $\Sl_2^1(\Z_p)$ generated by the conjuguates of $Z$ and put $G:=\Gamma/  \gamchap$. As $XZX^{-1}Z^{-1}=\left(\begin{array}{cc} 1 & p(1-e^{2p}) \\ 0 &1 \end{array} \right)$ becomes trivial in $G$,  one has that $X^p$ is trivial in $G$. Then thanks to  Proposition \ref{coro-generateurs}, one has: $G\simeq \Z/p\Z \times \Z/p\Z$.

\subsection{The group $\Sl_n^1(\Z_p)$}\label{sln}

Let $\sl_n(\Q_p)$ be the $\Q_p$-Lie algebra constitued by  the square matrices  $n\times n$ with  coefficients in  $\Q_p$ and of zero trace. It is a simple algebra of dimension $n^2-1$. 
Recall a natural basis of it:
\begin{enumerate} 
\item[(a)] for $i \neq j$, $E_{i,j}=(e_{k,l})_{k,l}$ for which all the coefficient are zero excepted  $e_{i,j}$ that takes value $p$;
\item[(b)] for $t>1$, $D_i=(d_{k,l})_{k,l}$ which is the diagonal matrix $D_i=(p,0,\cdots,0,-p,0,\cdots, 0)$, where $d_{i,i}=-p$. 
\end{enumerate}

Let $\sl_n$ be the $\Z_p$-Lie algebra generated by the $ E_{i,j}$ and the $D_i$. 
The algebra  $\sl_n$ is  powerful and uniform.

Put $X_{i,j}=\exp E_{i,j}$ and $Y_i=\exp D_i$. Denote by $\Sl_n^1(\Z_p)$ the subgroup of $\Sl_n(\Z_p)$ generated by the matrices $X_{i,j}$ and $Y_i$. The group $\Sl_n^1(\Z_p)$ is FAb, uniform  and of  dimension $n^2-1$. It is also the kernel of the reduction map  of  $\Sl_n(\Z_p)$ modulo $p$.

\medskip

By Seligman \cite{Seligman}, one knows that the automorphisms of the algebra $\sl_n(\Q_p)$ are {generated by those of two types, namely those} of the form   $\sigma_A(X)=X^{-1}AX$ with $A \in \Gl_n(\Q_p)$, or of the form $X\mapsto \sigma_A(-X')$, where $X'$ corresponds to the transposition of $X$.

\medskip

We now examine the inner automorphisms $\sigma_A$.
Put $\Gamma =\Sl_n^1(\Z_p)$ and take an automorphism $\sigma$ of order  $2$. The automorphism $\sigma$ induces an automorphism of order  $2$ on 
$\sl_n(\Q_p)$ that we assume of the form $\sigma_A$. As for  $\Sl_2^1(\Z_p)$, one reduces to the case   where $A$ has  $X^2-1$ or $X^2-\varepsilon$ as minimal polynomial, for some $\varepsilon \in \Z_p^\times \backslash (\Z_p^\times)^2$. In particular, when $n$ is odd, this polynomial is $X^2-1$. It is also the case if we have: $$\Gamma \rtimes \langle \sigma \rangle \simeq \Gamma \rtimes \langle A \rangle \hookrightarrow  \Gl_n(\Z_p),$$ where $A$ is a matrix of order $2$ that acts by conjugation. 
Let us assume that $A$ is of order~$2$. {As with the case of $\Sl_2^1(\Z_p)$, we can simplify to the case where $A$ is diagonal with $\pm 1$ eigenvalues.
 Denote by  $k=\dim \ker(A-I)$, \emph{i.e.} the number of $+1$s on the diagonal.}

\begin{lemm} With the above assumptions,  
the vector subspace $\big(\sl_n\big)_{\sigma_A}$ of the fixed points of the algebra $\sl_n$ under conjugation by $A$ is generated by the diagonal matrices and by the   matrices  $E_{i,j}$ for  $\{i,j\} \subset \{1,\cdots, k\}$ or for $\{i,j\} \subset \{k+1,\cdots, n\}$.
The matrices $E_{i,j}$ and $E_{j,i}$, with $i\leq k$ and $j>k$, form a basis of the subspace of the eigenvalue  $-1$.
\end{lemm} 

\begin{proof}
It is a simple computation.
\end{proof}

Denote by $H$ the subgroup of  $\Gamma$ generated by the matrices $X_{i,j}$ for $\{i,j\} \subset \{1,\cdots, k\}$ and for $\{i,j\} \subset \{k+1,\cdots, n\}$.

\begin{prop} Under the above conditions,
one has 
\begin{enumerate}
\item[(i)] $\displaystyle{\gamchapzero=\langle X_{i,j}, \{i,j\} \subset \{1,\cdots, k\}, \{i,j\}\subset \{k+1,\cdots, n\}, \   D_l, i\neq j, l>1 \rangle \cdot} $
\item[(ii)] $H \lhd {\gamchapzero}$ ;
\item[(iii)] $H \subset  \left(\begin{array}{c|c}A_{k,k}   &  0 \\  \hline
 0  & B_{n-k,n-k} \end{array} \right)$
\end{enumerate}
\end{prop}

\begin{proof} (i) is a consequence of  Proposition \ref{pointfixe-uniform}.

(ii) is an easy computation. 
\end{proof}

Hence one sees that the subgroup  $H$ is of dimension (as variety over $\Q_p$) at most $k^2+(n-k)^2$ which is strictly smaller than $n^2-1$.
On the other hand,  the quotient $\displaystyle{{\gamchapzero}/H}$ is generated by the diagonal  matrices, and is hence abelian; it will be finite if the subgroup  $\displaystyle{\gamchapzero}$ is open in $\Gamma$, because $\Gamma$ is FAb. Therefore $\Gamma$, which is of dimension $n^2-1$, is of the same dimension as $\gamchap$, which can not be of the same dimension as  $H$. Then $\displaystyle{\gamchapzero \subsetneq \gamchap}$, which proves that  the action of $\sigma$ on $\Gamma$ is \fpm.

\begin{prop} \label{prop-sln} Let $n\geq 2$ and let $\sigma=\sigma_A$ with  $\A \in \Gl_n(\Z_p)$ of order $2$. 
Then 
\begin{enumerate}
\item[(i)] $\Sl_n^1(\Z_p)/\Sl_n^1(\Z_p)_\sigma \simeq (\Z/p\Z)^{2 k(n-k)}$ ;
\item[(ii)]  The action of $\sigma$ on $\Sl_n^1(\Z_p)$  is \fpm.
\end{enumerate}
\end{prop}

\begin{proof}
One only has to verify (i): it is a computation as for $\Sl_2^1(\Z_p)$.
\end{proof}



\part{Arithmetic results} \label{PartII}

First, let us recall some notations. 
\begin{itemize}
\item  $p$ is a prime number.
\item If   $\K$  is a fixed number field, and if    $S$ and $T$ are  two finite and disjoints sets of primes ideals of $\O_K$, denote by  $\K_S^T$ the maximal pro-$p$ extension of  $\K$ unramified outside  $S$ and totally split at  $T$;  $\GST = \Gal(\K_S^T/\K)$.
\item
\emph{We assume throughout that  $S$ contains no primes above $p$} and  that for finite places  $\p\in S$, we have  $\# \O_\K/\p \equiv 1 (\bmod \ p).$  Hence, by class field theory, the pro-$p$ group  $\GST$ is {FAb}. 
\item Put $\Cl_S^T(\K):=\GST^{ab}$. It is the  $p$-Sylow of the $S$-ray  $T$-class class group of $\K$.
\item  
Let  $\O_{\K}^T$ be the group of $T$-units
  of~$\K$.
  \item Let $\O_{\K,S}^T$ be the subgroup of  $\O_\K^T$ defined as the kernel of the natural map  $$\O_\K^T \rightarrow \prod_{v\in S} \fq(\K,v)^\times,$$
 where  $\fq(\K,v)$ denote the residue field  of $\K$ at $v$; in other words, $\O_{\K,S}^T$ is the group of   $T$-units of  $\K$ congruent to   $1$ modulo $\p$, for all $\p \in S$.
\item If $\L/\K$ is an extension of  $\K$, we still denote by abuse,  $S=S(\L)$ be the set of primes of $\O_\L$ above the primes $\p\in S$.
\end{itemize}

\section{On the freeness of $T$-units in the non-semisimple case} \label{section-unités}

\subsection{The context} \label{thecontext}
\begin{itemize}
\item Let us start with a number field $\k$ with two finite and disjoints sets  $S$ and $T$ of primes of $\k$.  
\item Let $\L/\k$ be a finite Galois extension with Galois group $\G$. We assume that $\G$ has only one $p$-Sylow  subgroup $G$; put $\Delta:= \G/G$. Hence $\Delta$ is a finite group of order coprime to $p$. Put $\K=\L^G$. The group  $\O_{\L}^T$ of $T$-units of $\L$ has a structure of 
 $\Z[\G]$-module. 
\item Put $\E:=\O_\L^T/(\O_\L^T)^p=\fq_p\otimes \O_\L^T$.
\item Henceforth, we assume that: 
\begin{enumerate}
\item[$-$] the archimedean places of $\K$ split completely in $\L$,
\item[$-$]   the primes of  $T$ split completely in $\L/\k$, 
\item[$-$]  the extension  $\L/\K$ is unramified outside $S$.
\end{enumerate}
\item Denote by $S_{ram}$ the set of primes of primes of $\O_\K$ in $S$ that are ramified in $\L/\K$.  
\end{itemize}

\medskip
 
We are interested in finding some arithmetic situations for which the  $\fq_p[\G]$-module  $\E$ contains a non-trivial  free $\G$-module.
 To this end, we are going to use and extend  an idea of  Ozaki \cite{Ozaki}.
 
 \medskip

 Let us first recall the  semisimple version of Dirichlet's  unit theorem that will be of interest to us.

\begin{theo}[Dirichlet's unit theorem - see \cite{gras2}]\label{Dirichlet}
With all the notation and assumptions as listed above, let  $D_1,\cdots, D_m$ be the decomposition groups of the archimedean places of $\k$ in $\K/\k$.
Then one has the  isomorphism of  $\fq_p[\Delta]$-modules:
 $$\1\oplus  \O_\K^T/(\O_\K^T)^p  \simeq \Ind_{D_1}^\Delta \fq_p \oplus \cdots \oplus \Ind_{D_m}^\Delta \fq_p \oplus \fq_p[\Delta]^{|T|} \oplus \chi_p ,$$
 where   $\chi_p$ is the  cyclotomic character corresponding to the action  of $\Delta$ on the $p$th roots of the unity  in $\K$.
\end{theo}

\subsection{Technical results}

\subsubsection{On some structural elements}

Let $\Delta$ be a finite group of order coprime to $p$.

\begin{defi} If $M$ is a finitely generated $\fq_p[\Delta]$-module, denote by  $\r{M}$ its minimal number of  generators. 
\end{defi}

\begin{rema} The semisimplicity of $\fq_p[\Delta]$ allows us to study  $\r{M}$.
First, let us recall that an $\fq_p[\Delta]$-module $M$ is monogenic if and only if its character  $\chi_M$ is contained in the regular character $\Reg$ of $\fq_p[\Delta]$ (we write this as $\chi_M \leq \Reg$). Hence, given an  $\fq_p[\Delta]$-module $M$, to determine $\r{M}$  is equivalent to {resolving} the  decomposition of  $\chi_M$ into characters, all of them being in the regular representation.
\end{rema}

Denote by   $(\cdot)^*$   the  Pontryagin dual ${\Hom}(\ \cdot \ , \Q_p/\Z_p)$.
From now on, all the $\fq_p[\Delta]$-modules  under consideration are assumed finitely generated.

\begin{lemm}\label{rgsemisimple}  

(i)If    $A \twoheadrightarrow B$ is a surjective  $\fq_p[\Delta]$-morphism, then $\r{A} \geq \r{B}$. 

(ii) If $A$ is an $\fq_p[\Delta]$-module, then $\r{A}=\r{A^*}$.

(iii) If $A$ and $B$ are  $\fq_p[\Delta]$-modules such that    $A \hookrightarrow B$, then $\r{A} \leq \r{B}$.

(iv) Let $ \cdots \longrightarrow A \longrightarrow B \longrightarrow C \longrightarrow \cdots  $ be an exact sequence of   $\fq_p[\Delta]$-modules. Then  
 $$ \r{B}\leq \r{A}+\r{C} \ \cdot$$

(v) If $A$ and  $B$ are two  $\fq_p[\Delta]$-modules such that $B \simeq\fq_p[\Delta]^t \oplus A$, then $\r{B}=t+ \r{A}$. 

(vi) If $A$ and $B$ are two $\Delta$-modules such that  $pA=0$ and  $A \hookrightarrow B$, then $\r{A} \leq d_p B$.
\end{lemm}

\begin{proof}

(i) is obvious. 

For (ii), let us write $\chi_A=\chi_1+\cdots + \chi_s$  with  $s=\r{A}$ and  $\chi_i \leq \Reg$, $i=1,\cdots, s$. Then $\chi_{A^*}=\chi_A^*=  \chi_1^*+\cdots + \chi_s^*$, where $\chi_i^*$ is the dual character of  $\chi_i$ defined by $\chi^*(s)=\chi(s^{-1})$. But trivially,  $\chi_i^* \leq \Reg$ and then
$\r{A^*}\leq \r{A}$. It then implies,  $\r{A^{**}} \leq \r{A^*}$ and the conclusion holds by using the   $\Delta$-isomorphism $A\simeq A^{**}$.
 
 (iii) is a consequence of (i) and  (ii), after passing to the  dual.

 (iv) is consequence of  (i) and (iii).

For (v): it suffices to prove the equality for $t=1$. First of all,  one has   
$\r{B}\leq 1 + \r{A}$, \emph{i.e.} $\r{A} \geq \r{B}-1$.
Then let us write $\chi_B= \chi_1+\cdots \chi_s$, with $s=\r{B}$ and $\chi_i\leq \Reg$, $i=1,\cdots, s$. As $\chi_B$ contains the regular representation $\Reg$,  one may complete for example the  character $\chi_1$ with some irreducible characters coming from the other characters $\chi_i$, $i>1$, to obtain $\Reg$. Then we get $\chi_B = \Reg + \chi_2'+\cdots+ \chi_s'$, with for  $i=2,\cdots, s$, $\chi_i'\leq \chi_i\leq \Reg$. It implies that  $\Reg + \chi_A = \chi_B= \Reg+ \chi_2'+\cdots + \chi_s'$ showing  
$\chi_A= \chi_2'+\cdots + \chi_s'$ and then $\r{A} \leq s-1=\r{B}-1$. In conclusion, one has the  desired equality:  $\r{B}=\r{A}+1$.

(vi) It is a trivial estimation. It is clear that $\r{A} \leq d_p A$; the conclusion follows from the fact that   $d_p A \leq d_p B$.
\end{proof}

\medskip

\subsubsection{Semilocal rings and a lower bound} \label{section5.1.3}

We first recall some classical results about semilocal rings; for more details see  \cite[Chapter 2]{Lam}.

\medskip

Let us conserve the arithmetic context of the previous section \ref{thecontext}:  we start with $\G \simeq G \rtimes \Delta$, where   $G$ is the unique $p$-Sylow $G$ of $\G$.  

\medskip

The group algebra  $\fq_p[\G]$ is a semilocal ring of radical $\displaystyle{R:= \langle g- 1, \ g\in G \rangle}$;  the quotient  $\fq_p[\G]/R$ is isomorphic to the semisimple algebra $\fq_p[\Delta]$ (see \cite{CR}, \S 64, exercice).

\medskip

The algebra $\fq_p[\G]$ is also Frobenius: one has the  $\fq_p[\G]$-modules isomorphism $\fq_p[\G] \simeq \big(\fq_p[\G]\big)^*$, coming from the symmetric non-degenerate bilinear   form $\displaystyle{(f,g)=\sum_{g\in \G}f(h) g(h^{-1})}$, for  $\displaystyle{f=\sum_h f(h)h}$ and  $\displaystyle{g=\sum_h g(h) h} \in \fq_p[\G]$. Thus, every  free submodule $ M_0$ of a finitely generated $\fq_p[G]$-module  $ M$ is in direct factor in  $M$: indeed, if  $\fq_p[\G] \hookrightarrow M$ then by duality  $M^* \twoheadrightarrow \fq_p[\G]^*$; by  projectivity of $\fq_p[\G]$, the free module can be lifted in direct factor in $M^*$, and it suffices to take the dual again. 

\

In conclusion,  every  finitely generated  $\fq_p[\G]$-module  $M$ has a free maximal submodule  (in direct factor):  there exists an integer  $t=t(M)$ such that  $$M\simeq \fq_p[\G]^t\oplus N,$$ where $N$ is of torsion (for all  $ x \in N$, there exists $\lambda \in \fq_p[\G]$, $\lambda \neq 0$, such that $\lambda\cdot x =0$). The integer $t$ is unique.
We deduce a first relation: $$d_p M = |\G|t+ d_p N,$$
where here  $d_p$ means as usual the  $p$-rank.

On the other hand,  as $\fq_p[\G]^G\simeq \fq_p[\Delta]$, we get  $M^{G} \simeq \fq_p[\Delta]^t \oplus N^G$, and then by \ref{rgsemisimple} $(iv)$, $\r{M^G}=t+\r{N^G}$.

\medskip

 Recall that for a finitely generated $\fq_p[\G]$-module $A$, one has a $\Delta$-isomorphism: $(A^*)_G \simeq (A^G)^*$.  

\medskip

Recall now Nakayama's Lemma (see \cite[Chapter 2 \S 4]{Lam}).

\begin{lemm} \label{lemme-Nakayama}
The $\fq_p[\G]$-module $A$ is generated by  $m_1,\cdots, m_s$ if and only if, the  $\fq_p[\Delta]$-module $A/(R A)=A_G$ is generated by   $m_1 RA, \cdots, m_s RA$. In particular, one can take  $s=\r{A_G}$.
\end{lemm}

Let us come back to our context and start with $M=\fq_p[\G]^t\oplus N$. By Lemma \ref{rgsemisimple} (ii), $\r{N^G}=\r{(N^G)^*}$. By  Nakayama's Lemma \ref{lemme-Nakayama}, the $\fq_p[\G]$-module $N^*$ can be generated by \  $\r{(N^*)_G}=\r{(N^G)^*}=\r{N^G}= \r{M^G}-t$ \ elements. Then, let us remark that: $$d_p N^*\leq (|\G|-1) \big(\r{M^G}-t\big) \cdot$$
 Indeed, as $N^*$ is of torsion (because $N$ is), for every element $x \in N\backslash \{0\}$, we have $\fq_p[\G]\cdot x \simeq \fq_p[\G]/{\rm Ann}(x)$, where ${\rm Ann}(x)$ is the annulator of $x$, it is a non zero ideal of  $\fq_p[\G]$, and  then $d_p \fq_p[\G] \cdot x  = |\G| - d_p {\rm Ann}(x) \leq |\G|-1$. 
 Then $$d_p N =d_p N^* \leq \big(|\G|-1\big) \r{N^*},$$
 by Lemma \ref{lemme-Nakayama}.

 \begin{prop}\label{minoration-rang} Let $M$ be a finitely generated $\fq_p[\G]$-module. Then there is a well-defined non-negative integer $t=t(M)$ such that  $M\simeq \fq_p[\G]^t\oplus N$, with $N$ of torsion. For this $t$, we have the following lower bound:
$$t \geq  d_p M - \big(|\G|-1\big ) \cdot \r{M^G} \cdot$$
 \end{prop}
 
 \begin{proof}
 Note that   $d_p N= d_p M - t |\G|$ and that  $d_p N \leq (|\G|-1) \big(\r{\MM^G}-t\big)$: 
 the conclusion is  obvious.
 \end{proof}

\subsection{Cohomology of units}

We are going to apply  Proposition \ref{minoration-rang} to the  $\fq_p[\G]$-module $M=\E=\fq_p\otimes \O_{\L}^T$.
Our arithmetic context will allow us to give some situations where, {in the notation of the preceding proposition,}  $t$ may {be as big as possible} (thanks to the choice of $T$). To do this, we need to give a sharp estimate of $\r{\E^G}$.
More precisely, we propose to show the following result in this direction. 

\begin{theo}\label{constante-liberte}
{
As introduced at the beginning of this section,
let $\L/\k$ be a finite Galois extension with Galois group $\G$. We assume that $\G$ has only one $p$-Sylow  subgroup $G$; put $\Delta:= \G/G$. Hence $\Delta$ is a finite group of order coprime to $p$. Put $\K=\L^G$. We assume that the archimedean places of $\K$ split completely in $\L$.
There exists a constant $\A=\A(\L/\K) \in \Z$ (depending on   the "arithmetic" in $\L/\K$) such that if $m$ is any given positive integer, there exists a choice of a  set $T$ of size $|T|\leq m+A$ consisting of finite places of $\k$ that split completely in $\L/\k$, such that the $\fq_p[\G]$-module $\E=\fq_p\otimes \O_{\L}^T$ contains a  submodule isomorphic to $\fq_p[\G]^m$.
}
\end{theo} 

This theorem will be proved in two stages  (Theorem  \ref{liber} and Theorem \ref{liberté-inégalité-bis}) which will give an explicit formula for $\A(\L/\K)$ depending on whether $\L$ contains a primitve $p$th root of unity or not.  The proofs will occupy us in the next two subsections.

\begin{rema}
We will see that when $[\K:\Q]$ is small with respect to $|\G|$ then $\A >0$. But, as we will see too, the context of the cyclotomic $\Z_p$-extension produces some situations where $\A$ is negative.
\end{rema}

\begin{rema}
When the arithmetic context  $\L/\k$ is fixed, the growth  of $T$ allows us to  assure that  the $T$-units admit an arbitrarily large free $\fq_p[\G]$-submodule. This explains the appearance of the constant $\A$  in Theorems \ref{maintheo1bis}
and \ref{maintheorem0}.
\end{rema}

To achieve our goal, we are going to  develop  an idea of  Ozaki  \cite{Ozaki}.
Let us introduce a bit more notation. Let us write   $d_\infty$ for the number of archimedean places of  $\k$ that split completely in  $\K/\k$ and  $r_\infty$ the number of ramified archimedean places ({\emph{i.e.} those that are real in $\k$ and not real in $\K$). {We let $(r_1,r_2)$ be the signature of $\k$. We also let $\mu_p=\langle \zeta_p\rangle$ be the group of $p$th roots of unity. The proof will very much depend on whether non-trivial $p$th roots of unity are present in $\L$.}

\subsubsection{When $\O_\L^T$ does not contain $\mu_p$}

  We prove
 
\begin{theo} \label{liber} Let $\L/\K$ be a Galois extension with Galois group $\G \simeq G \rtimes \Delta$, where the order of $\Delta$ is coprime to the $p$-Sylow $G$ of $\G$. \emph{Suppose that $\L$ does not containt $\mu_p$}. Let us decompose the $\fq_p[\G]$-module $\E:= \fq_p \otimes \O_\L^T$ as $\fq_p[\G]^t\oplus N$, where $N$ is of  torsion.
Under the arithmetic conditions of section \ref{thecontext}, one has $$ t \geq |T|+ d_\infty +\frac{1}{2}r_\infty -(|\G|-1)\big(r_1+r_2-d_\infty-\frac{1}{2} r_\infty +d_p \Cl_S^T(\K)+|\G||S|+|S_{ram}|\big) -1 \cdot$$
\end{theo}

\begin{rema}\label{remaA0}
When $\mu_p$ is not contained in $\O_\L^T$, put $$\A=-\big(d_\infty +\frac{1}{2}r_\infty -(|\G|-1)\big(r_1+r_2-d_\infty-\frac{1}{2} r_\infty +d_p \Cl_S(\K)+|\G||S|+|S_{ram}|\big) -1\big).$$
{We note that $d_p \Cl_S^T(\K)\leq d_p \Cl_S(\K)$.}
\end{rema}

The goal of the end of this section  is to prove Theorem \ref{liber}.

\medskip
 
First of all, let us remark that all the cohomology groups  $H^\bullet(G, \cdot)$ are  finite $\Delta$-modules.

\medskip

Let us start with the exact sequence of  $G$-modules  corresponding to raising to the $p$th-power~: $$\1 \longrightarrow \O_\L^T \stackrel{p}{\longrightarrow} \O_\L^T \longrightarrow \E \longrightarrow 1 \cdot $$
We then have \begin{eqnarray}\label{suite-exacte-fonda} \fq_p\otimes \O_\K^T \hookrightarrow \E^G \twoheadrightarrow H^1(G,\O_\L^T)[p] \cdot\end{eqnarray}
Thanks to Lemma \ref{rgsemisimple} (iv), this exact sequence of $\fq_p[\Delta]$-module gives the inequality  \begin{eqnarray}\label{suite-exacte2} \r{\E^G} \leq \r{H^1(G,\O_\L^T)[p]}+\r{\fq_p\otimes \O_\K^T} \cdot \end{eqnarray}

\

We want to  bound the quantities $\r{H^1(G,\O_\L^T)[p]}$ and $\r{\fq_p\otimes \O_\K^T}$.

\

As the extension $\L/\K$ is unramified outside  $S$, it is possible to control the cohomology group $H^1(G,\O_\L^T)$. Indeed put $j_{\L/\K,S}^T : \Cl_{S}^T(\K) \rightarrow \Cl_{S}^T(\L)$ for the natural "conorm" morphism of ideal classes.
The following proposition, that is a generalization of a result of Iwasawa \cite{Iwasawa} when $T=S=\emptyset$,  has been proved in   Maire  \cite{Maire}:

\begin{prop}\label{prop-maire}
Let $\L/\K$ be a finite  Galois extension unramified outside $S$ and tamely ramified at  $S$. Put  $G=\Gal(\L/\K)$. Then $H^1(G,\O_{L,S}^T) \simeq_\Delta \ker (j_{L/K,S}^T)$.
\end{prop}

\begin{proof} See \cite{Maire}.
\end{proof}

\medskip

Let us estimate the difference between  $\O_\L^T$ and $\O_{L,S}^T$.
Start with the exact sequence $$1 \longrightarrow \O_{L,S}^T \longrightarrow \O_\L^T \longrightarrow B \longrightarrow 1,$$
where $\displaystyle{\varphi: B \hookrightarrow \prod_{v\in S} \fq(\L,v)^\times}$ and where  $\fq(\L,v)$
is the residue field of  $\L$ at $v$.
One has $$ \cdots \longrightarrow  H^1(G,\O_{\L,S}^T) \longrightarrow H^1(G,\O_\L^T) \longrightarrow H^1(G,B) \longrightarrow \cdots$$
and  $$\cdots \longrightarrow H^0(G,{\rm coker} \varphi) \longrightarrow H^1(G,B) \longrightarrow H^1(G,\prod_{v\in S} \fq(\L,v)^\times) \longrightarrow \cdots$$

Denote by $\alpha$ the morphism $\alpha: H^1(G,\O_\L^T)[p] \longrightarrow H^1(G,B)[p]$. Then, one has the 
 exact sequence of  $\fq_p[\Delta]$-modules: $$ 1 \longrightarrow \ker(\alpha) \longrightarrow H^1(G,\O_\L^T)[p] \longrightarrow \Im(\alpha) \longrightarrow 1 \cdot$$ 
By Lemma \ref{rgsemisimple} (iv), one has \begin{eqnarray} \label{inegalite10}\r{H^1(G,\O_L^T)[p]} \leq \r{\Im(\alpha)} + \r{\ker(\alpha)}. \end{eqnarray}

By Lemma \ref{rgsemisimple} (iii), as $\Im(\alpha) \hookrightarrow H^1(G,B)[p]$, one obtains \begin{eqnarray}\label{inegalite1110} \r{\Im(\alpha)} \leq \r{H^1(G,B)[p]}. \end{eqnarray} Moreover, if we denote   $B':=\Im\big(B^G \rightarrow H^1(G,\O_{\L,S}^T)\big)$, then  $$\ker(\alpha) \simeq \big(H^1(G,\O_{\L,S}^T)/B'\big)[p] \hookrightarrow H^1(G,\O_{\L,S}^T)/B' \twoheadleftarrow H^1(G, \O_{\L,S}^T)$$ and then, by Proposition \ref{prop-maire} and Lemma \ref{rgsemisimple} (vi), 
one obtains the upper bound:
\begin{eqnarray} \label{inegalite110} \r{\ker(\alpha)} \leq  d_p \Cl_S^T(\K).\end{eqnarray}
Using (\ref{inegalite10}), (\ref{inegalite1110}) and (\ref{inegalite110}),  one obtains:\begin{eqnarray}\label{inegalite11} \r{H^1(G,\O_\L^T)[p]}\leq \r{H^1(G,B)[p]}+ d_p \Cl_S^T(\K) \cdot\end{eqnarray}

\

To estimate $\r{H^1(G,B)[p]}$, 
let $\beta$  be the morphism $$\displaystyle{\beta: H^1(G,B)[p] \longrightarrow \big(H^1(G,\prod_{v\in S} \fq(\L,v)^\times)\big)[p]}.$$ One has
\begin{eqnarray}\label{inegalite20} \r{H^1(G,B)[p]} \leq \r{\Im(\beta)}+\r{\ker(\beta)}.\end{eqnarray} 
As before, the  evaluation of  $\r{\Im(\beta)}$ is made thanks to the estimation of $$\displaystyle{\r{H^1(G,\prod_{v\in S} \fq(\L,v)^\times)[p]}}.$$
By Shapiro's Lemma, one has the $\Delta$-isomorphism $$\displaystyle{H^1(G,\prod_{v\in S} \fq(\L,v)^\times) \simeq \bigoplus_{v\in S} H^1(G_v, \fq(\L,v)^\times)\cdot}$$
Let us fix  $v\in S$. Let $G_v=\Gal(\L_v/\K_v)$ be the  decomposition group in $\L/\K$ of a place of $\L$ above $v$; put  $I_v$ the inertia group  $I_v$ associated to $G_v$. The quotient $G_v/I_v$ corresponds to the Galois group of the local  maximal unramified extension   $\L_v^{nr}/\K_v$ in  $\L_v/\K_v$.  The quotient  $G_v/I_v$ is isomorphic to   the Galois group of the associated residual extensions; as $\fq(\L,v)$ is isomorphic to the residual field of  $\L_v^{nr}$, By Hilbert Theorem 90, one has:  $H^1(G_v/I_v, \fq(\L,v)^\times)=0$.
Hence the exact sequence $1\longrightarrow I_v \longrightarrow G_v \longrightarrow G_v/I_v \longrightarrow 1$ induces 
$$1 \longrightarrow H^1(G_v,\fq(\L,v)^\times) \longrightarrow H^1(I_v,\fq(\L,v)^\times)^{G_v/I_v} \longrightarrow \cdots $$
The finite group $I_v$ acts trivially on  the cyclic group $\fq(\L,v)^\times$, then $H^1(I_v,\fq(\L,v)^\times)$ is cyclic because the ramification of  $v$ is tame.
To resume,  $$\displaystyle{H^1(G,\prod_{v\in S} \fq(\L,v)^\times)  \simeq \bigoplus_{v\in S_{ram}(\K)} C_v } ,$$
where $S_{ram}(\K)$ is the set of places of  $S(\K)$ that are ramified in $\L/\K$ and where $C_v$ is a cyclic group with $C_v:=\ker\big(H^1(I_v,\fq(\L,v)^\times) \rightarrow H^2(G_v/I_v,\fq(\L,v)^\times)\big)$.
 
 For  $v_0 \in S(\k)$, the group $\Delta$ acts transitively on $\displaystyle{\prod_{v|v_0} C_v}$, where the produce is on the  places $v$ of $\K$ above $v_0$. Then  $$\r{H^1(G,\prod_{v\in S} \fq(\L,v)^\times)[p]} \leq \sum_{v_0 \in S_{ram(\k)}} \r{\prod_{v|v_0} C_v[p]} \leq   |S_{ram}|,$$
where here $S_{ram}=S_{ram}(\k)$ is the subset of places of  $S$ ($=S(\k)$) ramified in $\L/\K$. 
Finally \begin{eqnarray}\label{inegalite21}\r{\Im(\beta)} \leq |S_{ram}|.\end{eqnarray}

To control  $\ker(\beta)$, as before, one uses the  estimation of Lemma \ref{rgsemisimple} (vi) to obtain $\r{\ker(\beta)} \leq d_p \big({\rm coker} \varphi\big)^G$  and then:
\begin{eqnarray}\label{inegalite22}  d_p \big({\rm coker} \varphi\big)^G \leq d_p {\rm coker} \varphi \leq d_p \big(\prod_{v\in S(\L)} \fq(\L,v)^\times\big) \leq |S(\L)| \leq |S||\G|  \cdot \end{eqnarray}

\medskip

By using (\ref{inegalite11}), (\ref{inegalite20}), (\ref{inegalite21}), and (\ref{inegalite22}):$$\begin{array}{rcl} \displaystyle{\r{H^1(G,\O_\L^T)[p]} }  &\leq &  \displaystyle{\r{H^1(G,B)[p]}+d_p \Cl_S^T(\K)} \\ 
&\leq &  |S||\G|  +|S_{ram}|+ d_p \Cl_S^T(\K) ,\end{array}$$
where we put  $S=S(\k)$ and $S_{ram}=S_{ram}(\k)$.
Now it suffices to use Dirichlet's unit Theorem \ref{Dirichlet} to obtain:
$$\r{\E^G}\leq \r{H^1(G,\O_\L^T)[p]} + \r{\fq_p\otimes \O_\K^T} \leq |\G| |S|+|S_{ram}|+d_p\Cl_S^T(\K) + r_1 + r_2 -1 + |T|\cdot$$
We finish the proof of Theorem \ref{liber} thanks to Proposition \ref{minoration-rang} and to the fact that $$d_p \E= |\G| \big(\frac{1}{2} r_\infty+ d_\infty+|T|\big)-1 \cdot$$

\subsubsection{When $\O_\L^T$ contains $\mu_p$}
As for the previous section, let us write the $\fq_p[\G]$-module $\E:= \fq_p \otimes \O_\L^T$ as  $\fq_p[\G]^t\oplus N$, where $N$ is of  torsion. 

\begin{theo} \label{liberté-inégalité-bis}
Let us consider the same context as for Theorem \ref{liber}, with the exception that $\L$ contains $\mu_p$. Let us decompose the $\fq_p[\G]$-module $\E:= \fq_p \otimes \O_\L^T$ as $\fq_p[\G]^t\oplus N$, where $N$ is of  torsion. Then
 $$ t \geq |T| + d_\infty +\frac{1}{2}r_\infty -(|\G|-1)\big(r_1+r_2+1-d_\infty-\frac{1}{2} r_\infty +d_p \Cl_S^T(\K)+|\G| |S|+|\G| |S_{ram}|+ d_p H^2(G,\fq_p) \big)  \cdot$$
\end{theo}

\begin{rema} \label{remaA1} When $\mu_p \subset \O_\L^T$, put $$\A= -\big[d_\infty +\frac{1}{2}r_\infty -(|\G|-1)\big(r_1+r_2+1-d_\infty-\frac{1}{2} r_\infty +d_p \Cl_S(\K)+|\G| |S|+|\G| |S_{ram}|+ d_p H^2(G,\fq_p) \big) \big].$$
{As before, we note that $d_p \Cl_S^T(\K) \leq d_p \Cl_S(\K)$.}
\end{rema}

Let us start with the following sequence of   $G$-modules  \begin{eqnarray}\label{suite-exacte2} 1 \longrightarrow \O_\L^T/ \mu_p \stackrel{p}{\longrightarrow} \O_\L^T \longrightarrow \E \longrightarrow 1  \ \cdot \end{eqnarray}

Then (\ref{suite-exacte2}) becomes
\begin{eqnarray} \label{2nde-suite-exacte-fonda}  \fq_p\otimes \O_\K^T \hookrightarrow \E^G \longrightarrow H^1(G,\O_K^T/\mu_p) \longrightarrow \cdots   \end{eqnarray}

Consider $$1 \longrightarrow  \mu_p \longrightarrow \O_\L^T {\longrightarrow} \O_\L^T / \mu_p \longrightarrow 1  $$
which gives
$$ \cdots \longrightarrow H^1(G,\mu_p) \longrightarrow H^1(G,\O_\L^T)  \longrightarrow H^1(G, \O_\L^T/\mu_p) \longrightarrow H^2(G,\mu_p) \longrightarrow H^2(G,\O_\L^T) \longrightarrow \cdots $$
Let us remark here that the $p$-group $G$ acts trivially on $\mu_p$. Thus  for $i=1,2$, the groups $H^i(G,\mu_p)$ describe generators and relations of $G$.

One then has $$ d_p H^1(G,\O_\L^T/\mu_p) \leq d_p H^2(G,\fq_p)+ d_p H^1(G,\O_{\L,S}^T)\cdot$$
and $$\begin{array}{rcl} \r{\E^G} &\leq& \r{\fq_p\otimes \O_\K^T}+ d_p H^1(G,\O_\L^T/\mu_p) \\
& \leq & \r{\fq_p\otimes \O_K^T} + d_p H^1(G,\O_\L^T)+d_p H^2(G,\fq_p) \\
& \leq & \r{\fq_p\otimes \O_K^T} + |\G||S|+|\G| |S_{ram}| + d_p \Cl_S^T(\K)+d_pH^2(G,\fq_p)
\end{array} $$
where for the last inequality, one takes the previous computation concerning $H^1(G,\O_\L^T)$ to obtain an upper bound  for $d_p H^1(G,\O_\L^T)$. 
The conclusion may be deduced from Proposition \ref{minoration-rang}.

\section{Ramification with prescribed Galois action}

\subsection{Preparation} \label{rappel-kummer} 

\subsubsection{Kummer Theory} Our reference here is the book of Gras \cite{gras}, \S 6, chapter I.

 Let us start with a Galois extension $\L/\k$ of Galois group $\G$ and recall some notations.

\begin{itemize}
\item Denote by $\chi_p=\fq_p(1)$ the  cyclotomic character resulting from the action on the $p$th roots of unity.  For a $\fq_p[\G]$-module $M$, put $M(1)=M\otimes_{\fq_p} \fq_p(1)$.
\item
 Let  $T$ be a finite set of primes of $\O_\k$ \emph{all of which split completely in} $\L$, and consider $$\V^T=\{ \alpha \in \L^\times, \ v_\P(\alpha)\equiv 0 \ (\bmod p), \forall \O_L\text{-primes } \P|\p \notin  T\}\cdot$$ 
\item  
Consider now the governing field $\F^T:=\L'(\sqrt[p]{\V^T})$, where $\L'=\L(\zeta_p)$. The Kummer extension  $\F^T/\L'$ is unramified outside  $T\cup S_p(\L')$.
\item We also define analogous objects over $\k$, namely: $$\V_\k^T=\{ \alpha \in \\k^\times, \ v_\P(\alpha)\equiv 0 \ (\bmod p),  \forall \O_L\text{-primes } \P|\p  \notin  T\}\},$$ and  the governing field  $\F_\k^T:=\k(\zeta_p,\sqrt[p]{\V_\k^T})$.
 \end{itemize}
 
 \begin{rema} One easily see that  $\fq_p\otimes \O_\L^T \hookrightarrow  \V^T/(\L^\times)^p$. {We will be interested in finding some }free sub-$\fq_p[\G]$-modules of $\V^T/(\L^\times)^p$: they will appear thanks to control over the group of $T$-units   $\fq_p\otimes \O_\L^T$ in conjunction with Theorem \ref{constante-liberte}.
\end{rema}

Put $\H=\Gal(\F^T/\L')$; the group $\G$ acts on $\V^T/(\L^*)^p$ and then on  $\H$.
Recall that the bilinear form $$\begin{array}{rcl} b : \V^T/(\L^*)^p \times \H & \longrightarrow & \mu_p \\
(x,h) &\mapsto & \sqrt[p]{x}^{h-1}
\end{array}$$ is non-degenerate and functorial with respect to the action of $\G$ : $$b_\varepsilon(g(x),h)=b_\varepsilon(x,g^{-1}(h)^{\chi_p(g)}).$$
This bilinear form induces an isomorphism of    $\G$-modules~: 
\begin{eqnarray}\label{iso} \Theta :  \big(\V^T/(\L^\times)^p\big)^*(1) \stackrel{\approxeq}{\longrightarrow} \H  \ \cdot\end{eqnarray}

\begin{rema} Put $M= \V^T/(\L^\times)^p$.
The bilinear form induces an isomorphism between    $\H$ and $\Hom(M,\mu_p)$, the latter  being isomorphic to  $\Hom(M,\fq_p)\otimes \mu_p = M^*(1)$. 
\end{rema}

\begin{prop}\label{liberté}
If the  $\fq_p[\G]$-module $\fq_p\otimes \O_\L^T$ contains  a free submodule $\langle  \varepsilon \rangle_\G$ generated by the unit $\varepsilon$, then $\H=\Gal(\F^T/\L')$ contains, as a direct factor, a free sub-$\fq_p[\G]$-module  $\H_{\varepsilon}$ of rank $1$, isomorphic to  $\big(\langle \varepsilon \rangle_\G\big) ^*(1)$, the latter being isomorphic to  $\H_\varepsilon:=\Gal(\L'(\sqrt[p]{\langle\varepsilon \rangle})/\L')$.
\end{prop}

\begin{proof}  As $\langle \varepsilon \rangle_\G$ is free, it is a direct factor in  $\V^T/(\L^\times)^p$.
By passing to the dual, the module $\big(\langle\varepsilon \rangle_\G\big)^*(1)$ is free and is in direct factor in  $\big(\V^T/(\L^\times)^p\big)^*(1) \stackrel{\Theta}{\Mapsto} \H=\Gal(\F^T/\L')$.
Finally by Kummer theory, $\big(\langle \varepsilon \rangle_\G\big)^*(1) \simeq \H_\varepsilon$. 
\end{proof}

\begin{defi}
 Under the hypothesis of Proposition \ref{liberté},
denote by $x_\varepsilon$ a generator of the free module $\H_\varepsilon$.
 \end{defi}

\subsubsection{The Theorem of Gras-Munnier}

 \begin{defi} Let $\K$ be a number field and $S$ a finite set of prime ideals of $\O_\K$. 
 {We say the extension $\L/\K$ is $S$-ramified if it is unramified outside $S$ and $S$-totally ramified if it is $S$-ramified and moreover all primes in $S$ are totally ramified in  $\L/\K$.}

\end{defi}

Let us conserve the notation introduced in the beginning of this  section \ref{rappel-kummer}: $\L'=\L(\zeta_p)$ and $\F^T=\L'(\sqrt[p]{\V^T})$.
Let us recall the   Theorem of  Gras-Munnier (see \cite{Gras-Munnier}, \cite{gras})
that will be extremely useful to us.

\begin{theo}[Gras-Munnier \cite{Gras-Munnier}]\label{gras-munnier}
Let $S=\{\p_1,\cdots, \p_m\}$ and $T$ be two finite sets of prime ideals of $\O_\L$, such that $S\cap T=\emptyset$,
and such that for all $\p_i \in S$, $\N \p_i \equiv 1 (\bmod p)$. 
For each $i=1,\cdots,m$, let  $\P_i$ be a prime of $\O_{\L'}$ above  $\p_i$.
Then, there exists a $T$-split, $S$-totally ramified cyclic extension $\F/\L$
of degree  $p$ if and only if, for $i=1,\cdots, m$, there exists $a_i \in \fq_p^\times$, such that
$$\prod_{i=1}^m \FF{\F^T/\L'}{\P_i}^{a_i} =1 \ \in \Gal(\F^T/\L'),$$
where  $\FF{\F^T/\L'}{ \bullet }$ is the Artin symbol in the extension $\F^T/\L'$.
\end{theo}

Note that the condition does not depend on the choice of the primes $\P_i$
above $\p_i$ (which merely causes a shift in the exponents $a_i$).

\subsubsection{Chebotarev density Theorem and applications}\label{chebotarev-consequences}

The Chebotarev density Theorem allows us to give a relationship between the Theorem of Gras-Munnier and the section about Kummer Theory.  
We continue to conserve the notations and the context of section   \ref{rappel-kummer}.

\begin{defi} 
Let $U$, $S$ and $T$ be three pairwise disjoint sets of  prime ideals of  $\O_\L$. Put $\Sigma=S\cup U$ and  assume that $\Sigma$ is {tame, i.e.\ $(\Sigma,p)=1$}. Denote by  $\I_S^T(U,\L)$ the subgroup of $\GG_\Sigma^T(\L)/\Phi(\GG_\Sigma^T(\L))$ generated by the inertia groups of the  prime ideals  of   $U$.
\end{defi}

\begin{lemm} \label{remarque-ramification0}With notation as above, the following conditions are equivalent.
\begin{itemize}
\item $\I_S^T(U,\L)=\{1\}$
\item Every $T$-split $\Sigma$-ramified cyclic degree $p$ extension of $\L$ is $S$-ramified
\item For every non-empty subset $U'$ of $U$, there does not exist a cyclic degree $p$ $T$-split $U'\cup S$-ramified  extension of $\L$ where all primes of  $U'$ are totally ramified.
\end{itemize}
\end{lemm}

\begin{proof} Obvious.  \end{proof}

\medskip

\begin{coro} \label{corollaire-structure}
Suppose that the $\fq_p[\G]$-module $\Gal(\F^T/\L')$ contains a free submodule  $\H_\varepsilon=\langle x_\varepsilon \rangle_\G$ of rank  $1$. By Chebotarev density Theorem, choose   a  prime ideal $\p$ of $\O_\L$ such that  
$\langle \FF{\F^T/\L'}{\P}\rangle =\langle x_\varepsilon \rangle $, where $\P|\p$. Then for  $U=\{ g(\P)=\P^g, \ g \in \G\}$, we have $\I^T(U,\L)=\{1\}$.
\end{coro}

\begin{proof}
Let first recall the property of the Artin symbol: for $g\in \G$ and $\P \subset \O_\L$, one has:
$$\displaystyle{\FF{\F^T/\L'}{\P^g}=g\FF{\F^T/\L'}{\P}g^{-1}=\FF{\F^T/\L'}{\P}^{g^{-1}}}\cdot $$ {By hypothesis there does not exist a non-trivial relation}  between the conjugates of  $\displaystyle{\FF{\F^T/\L'}{\P}^g}$ with $g\in \G$. Then by   Theorem \ref{gras-munnier}, there exists no $T$-split, $U'$-totally ramified, cyclic degree $p$ extension of $\K$,  for every non empty set $U'$ of $S$, meaning that  $\I^T(U,\L)$ is trivial (thanks to  Lemma \ref{remarque-ramification0} with $S=\emptyset$).
\end{proof}

In fact, we want to say more.  For a finite set  $S$, $\G$-stable, of tame ideal primes of  $\O_\L$ with $S\cap T=\emptyset$, denote by  $\F(S)$ the subgroup of  $\Gal(\F^T/\L')$ generated by the  Frobenius of the ideals of $S$ (with an abuse of notation); here the primes in $S$ are unramified in $\F^T/\L'$. 

\begin{coro}\label{coro-evitement} Suppose that the  $\fq_p[\G]$-module $\Gal(\F^T/\L')$ contains a free submodule  $\H_\varepsilon=\langle x_\varepsilon \rangle_\G$ of rank $1$  such that $${\H_\varepsilon} \  \bigcap \ \F(S)= \{0\}\cdot$$ By Chebotarev density Theorem, choose a prime ideal   $\p$ of $\O_\L$  such that   
$\langle \FF{\F^T/\L'}{\P} \rangle =\langle x_\varepsilon \rangle$, for any $\P|\p$. Put $U=\{ g(\P)=\P^g, \ g \in \G\}$. Then  $\I_S^T(U,\L)=\{1\}$.
\end{coro}

\begin{proof}  Let $\L_0/\L$ be a $T$-split, $S\cup U$-ramified, degree $p$ cyclic  extension of $\L$.
As the  free $\fq_p[\G]$-module  $\langle \displaystyle{\FF{\F^T/\L'}{\P}}\rangle_\G $ intersects trivially  $\F(S)$, one has thanks to Theorem  \ref{gras-munnier} that the extension $\L_0/\L$ is unramified at  $U$. By Lemma \ref{remarque-ramification0},  one concludes that $\I^T_S(U,\L)=\{1\}$. 
\end{proof}

\subsection{The set $\SS$}

 \label{kummer0}

We are now going  to give a non free situation that will be used  in the proof of Theorem  \ref{maintheorem}. It is essential for the definition of the sets $\SS$.

Let us start from the existence of a free submodule $\fq_p[\G]^{|\G|}$ of $\V^T/(\L^\times)^p$, of rank  $|\G|$. Let  $(\varepsilon_g)_g$ be a basis of $\fq_p[\G]^{|\G|}$ indexed by the elements of  $\G$.

As $\fq_p[\G]$ is a Frobenius ring, the free module $\displaystyle{\bigoplus_{g\in \G} \fq_p[\G] \varepsilon_g}$ is a direct factor in $\V_\L^T/(\L^\times)^p $; put then  $$\V_\L^T/(\L^\times)^p = \bigoplus_{g\in\G}\fq_p[\G]  \varepsilon_g \oplus W,$$as the  sum of  $\G$-modules.

\medskip

 Let $\displaystyle{N=\sum_{h\in \G} h}$  be the algebraic norm.  Let us mention an easy lemma:
 
 \begin{lemm}\label{calcul-trivial} The module
  $\fq_p N$ is a sub-$\fq_p[\G]$-module of $\fq_p[\G]$ generated by $N$. In other words, $\langle N \rangle_\G = \langle N \rangle$.
 It is also the only sub-$\G$-module  of $\fq_p[\G]$ on which  $\G$ acts trivially.
 \end{lemm}
 
 \begin{proof}
 Put $\displaystyle{\sum_{g\in \G} a_g g \in \fq_p[\G]}$, $a_g \in \fq_p$. Then
 $$\sum_{g\in \G} a_g g\big(\sum_{h\in \G} h \big)= \sum_{g \in \G} a_g \sum_{h \in \G} gh=\sum_{g\in \G} a_g N \in \fq_p N , $$
which proves the first part. 
Now clearly   $\G$ acts trivially on $N$  and moreover if we start with an  element $\displaystyle{\sum_{g\in \G} a_g g }$ on which  $\G$ acts trivially, then obviously,  $a_g$ is constant (not depending on $g \in \G$).
 \end{proof}
 
 \medskip
 
Take   $\varepsilon_0 \in \V_\k^T(\L^\times)^p/(\L^\times)^p$ and write $\displaystyle{\varepsilon_0=\big(\sum_{g\in \G} y_g\big) +z}$, with $y_g\in \fq_p[\G] \varepsilon_g$ and $z\in W$.  As  $\G$ acts trivially  on $\varepsilon_0$, then  $\G$ acts trivially on the elements $y_g$ and Lemma \ref{calcul-trivial} shows that  $y_g \in  \fq_p N \cdot \varepsilon_g$. Denote by abuse, $\langle N \rangle :=\fq_p N \cdot \varepsilon_g$.

The morphism of $\fq_p[\G]$-modules $$\V_\L^T/(\L^\times)^p\twoheadrightarrow \bigoplus_{g\in \G} \big( \fq_p[\G]  \varepsilon_g / \langle N \rangle \big) $$ factors through  $\V_\k^T(\L^\times)^p/(\L^\times)^p$. Passing to the dual, on obtains: 
$$  \big(\bigoplus_{g\in \G} \fq_p[\G] \varepsilon_g / \langle N \rangle \big)^*(1) \hookrightarrow \big(\V_\L^T/(\L^\times)^p\V_\k^T\big)^*(1) $$ where   $$\big(\V_\L^T/(\L^\times)^p\V_\k^T\big)^*(1) = \ker\big[ \big(\V_\L^T/(\L^\times)^p\big)^*(1) \twoheadrightarrow \big(\V_\k^T(\L^\times)^p/(\L^\times)^p\big)^*(1)\big]\cdot$$

By passing to Kummer theory and by using  the isomorphism $\Theta$ of (\ref{iso}), we get:
$$\xymatrix{ 0 \ar[r] & \big(\V_\L^T/(\L^\times)^p\V_k^T\big)^*(1)\ar[d]^\approxeq  \ar[r] &  \big(\V_\L^T/(\L^\times)^p\big)^*(1) \ar[d]^\approxeq_\Theta \ar[r] & 
 \big(\V_\k^T (\L^\times)^p/(\L^\times)^p\big)^*(1) \ar[d]^\approxeq \ar[r] & 0 \\
  0 \ar[r] & \Gal(\F^T/\F_\k^T\L') \ar[r] & \Gal(\F^T/\L') \ar[r] & 
 \Gal(\F_\k^T\L'/\L') \ar[r] & 0 
 }$$ 
Put   \begin{eqnarray}\label{sous-espace} \displaystyle{\H':= \Theta\big( \big(\bigoplus_{g\in \G} \fq_p[\G] \varepsilon_g / \langle N \rangle \big)^*(1)  \big) };\end{eqnarray} then  $\H' \subset \Gal(\F^T/\F_\k^T\L')$.

\medskip

Let us study more carefully  $\H'$. First, 
  by Kummer duality, one has $$\bigoplus_{g\in \G}  \big( \fq_p[\G] \varepsilon_g / \langle \N \rangle  \big)^* \hookrightarrow \bigoplus_{g\in \G} \big(\fq_p[\G]\varepsilon_g\big)^* \twoheadrightarrow \bigoplus_{g\in \G}\langle N\rangle ^*\cdot$$
We will continue to denote  by  $(\varepsilon_g)_g$ the dual basis of  $\varepsilon_g$.

  \
  
 Let us fix an element $\varepsilon_g$. 
Then $\big(\fq_p[\G]/\langle N \rangle \big)^* \simeq  \{f \in \Hom(\fq_p[\G],\fq_p), \ f(N)=0\}$, see for example \cite{CR}, \S 60, chapter IX. Let  $$\I=\ker \big( \fq_p[\G] \rightarrow \fq_p\big)$$  be the augmentation  ideal of the algebra $\fq_p[\G]$. 
 Obviously,  \emph{via}  the isomorphism between $\fq_p[\G]^*$ and $\fq_p[\G]$, one has $\I \subset  \{f \in \Hom(\fq_p[\G],\fq_p), \ f(N)=0\}$; these two  $\fq_p$-spaces vector have the  same dimension, \emph{i.e.} $|\G|-1$, and then finally   $\I= \{f \in \Hom(\fq_p[\G],\fq_p), \ f(N)=0\}$. The exact sequences  
 $$ 1\longrightarrow \langle N \rangle \longrightarrow \fq_p[\G] \longrightarrow \fq_p[\G]/\langle N \rangle  \longrightarrow 1$$
and  $$ 1\longrightarrow \I \longrightarrow \fq_p[\G] \longrightarrow \fq_p \longrightarrow 1$$
are dual to each other, and the same holds after tensoring by $\mu_p$. 

\medskip

Put  $x_g=\varepsilon_g \otimes \zeta_p$: it is a generator of the free module $\big(\fq_p[\G]\varepsilon_g\big)(1)$.
 In the sum $\displaystyle{\bigoplus_{g\in \G} \I \cdot x_g  } \hookrightarrow \bigoplus_{g \in \G} \fq_p[\G] x_g$, let us choose the  particular element  $x$ defined by  \begin{eqnarray}\label{elementclef} x:= \big(\sum_{g\in \G} (g-1) x_g\big)\cdot \end{eqnarray}

Obviously the algebraic norm kills each component  $g-1$ of $x_g$ and then $N(x)=0$. In fact:

\begin{lemm}\label{relation-triviale}
The relation $N(x)=0$ is the unique non trivial relation  of $x$, \emph{i.e} if $\displaystyle{\sum_{h\in \G} a_h h \cdot x=0}$ then $a_{h}=a_e$ for all $h\in \G$. Equivalently,  ${\rm Ann}(x)=  \fq_p N $.
\end{lemm}

\begin{proof} 
Write  $\displaystyle{\lambda= \sum_{h\in \G} a_h h \in \fq_p[\G]}$ such that $\lambda \cdot x=0$. Then 
$$0 = \lambda x = \sum_{g\in \G}  \lambda (g-1) x_g \cdot $$
As the modules $\langle x_g \rangle $ are in direct factor, one has for every $g\in \G$, $\lambda  (g-1) x_g=0$. The modules $\langle x_g \rangle $ being moreover free, one gets  $\lambda (g-1)=0$. Thus  $\displaystyle{\lambda \in \bigcap_{g\in \G} {\rm Ann}(g-1)} \in \fq_p[\G]$. 
To conclude, it suffices to remark that the intersection is reduced to $(N)=\fq_p N$. 
Indeed, when  $g$ is fixed, we get $\displaystyle{\sum_{h\in \G} a_h h(g-1)=0}$ if and only if,  $a_{h^{-1}g}=a_g$ for all $h$. 
When varying $g$, one obtains  $a_{hg}=a_g$ for all $h$ and $g$, implying  $a_g=a_e$ for all $g\in \G$.
\end{proof}

\subsection{Some consequences}

Let us start now with  $x$ given by Definition (\ref{elementclef}). 

Recall that $\displaystyle{ x\in \bigoplus_{g\in \G}\I x_g}$, where $\I=\{f \in \Hom(\fq_p[\G],\fq_p), f(N)=0\}$. 

Put  $x_0= \Theta(x) \in \Gal(\F^T/\L')$, where $\Theta$ is the isomorphism coming from Kummer theory, see (\ref{iso}).
The element  $x_0 $ is in  $\H'$ and then $x_0 \in \Gal(\F^T/\F^T_k\L')$.

\medskip

By Chebotarev density Theorem, let us choose  a prime ideal $\P$ of $\O_\L$ which splits totally in  $\L/\k$ and such that $\langle \FF{\F^T/\L'}{\P}\rangle = \langle x_0\rangle $.

Let  $\p_\k={\rm N}(\p)={\rm N}_{\L/\k}(\P)$ be  the unique prime ideal of  $\O_\k$ under $\p$. 
Put  $U=\{\p_\k\}$ and still denote by abuse $U=U(\F)=\{ \P \subset \O_\F, \ \P | \p_\k\}$ when $\F/\k$ is a finite extension.

\begin{rema} When  $S=\emptyset$ and $s=1$, in the main theorems (Theorems \ref{maintheo1bis} 
and \ref{maintheorem0}
) the set $\SS$ considered is composed of such prime ideals. The set $\SS$ is of positive density. This density  depends on  the discriminant of $\F^T/\Q$ and on the size of $\Gal(\F^T/\Q)$.
The discriminant of $\L'/\Q$ is related to the number field $\K$; the discriminant of  $\F^T/\L'$ depends
on the wild ramification in  $\F^T/\L'$ and on the tame ramification at $T$;  and the size of  $\Gal(\F^T/\L')$ depends the $p$-class group of $\K$, on the  signature of $\K$ and on the size of $|T|$.  
\end{rema}

\begin{prop} \label{lemmepreparatoire} 
With  the previous notations and conditions (especially the choice of $\P$),  we get the isomorphism of $\G$-modules:  $\I^T(U,\L) \simeq \I^T(U,\k) \simeq \fq_p$.\end{prop}

\begin{proof}
 Suppose that there exists a non-trivial relation between the conjuguates  $\displaystyle{\FF{\F^T/\L'}{\P}}^g$, $g\in \G$, of $\displaystyle{\FF{\F^T/\L'}{\P}}$: $\displaystyle{\big(\sum_{g\in \G} a_g g\big)\cdot \FF{\F^T/\L'}{\P}=0},$
with $a_{g_0}\neq 0$ for at least one  $g_0\in \G$.  Then, as $\langle \FF{\F^T/\L'}{\P}\rangle = \langle x_0\rangle $, by 
Lemma  \ref{relation-triviale},  one gets $a_g =a_{g_0} \neq 0$ for all $g\in \G$. Thus by Theorem \ref{gras-munnier},   every $T$-split   degree $p$ cyclic  extension  of $\L$ which is ramified at one prime  $\P_0|\p$ is totally ramified  at all $\P_0^g$, $g\in \G$. That means that   $d_p \I^T(U) \leq 1$  (it is an easy 
generalization of Lemma \ref{remarque-ramification0}). 

We now show that the number field  $\k$ has a $T$-split, $\{\p_\k\}$-totally ramified,  degree  $p$ cyclic extension. 
Indeed, by the choice of $\P$, one knows that $\FF{\F^T/\L'}{\P} \in \langle x_0 \rangle  \subset  \H'$ and consequently, $\displaystyle{\FF{\L'\F_\k^T/\L'}{\P} =1}$. By the properties of the  Artin symbol, one gets 
$$\FF{\L'\F_\k^T/\k'}{{\rm N}_{\L'/\k'}(\P)}=\FF{\L'\F_\k^T/\L'}{\P}=1,$$
where $\FF{\F_\k^T/\k'}{{\rm N}_{\L'/\k'}(\P)}=1$. We then remark that ${\rm N}_{\L'/\k'}(\P)$ is a prime ideal of $\O_\k$ above $\p$. By Theorem \ref{gras-munnier}, it proves the existence of a $T$-split, $\{\p_\k\}$-totally ramified, degree $p$ cyclic extension of $\k$.
Then,  $\I^T(U,\k)\simeq \fq_p$ as  $\G$-modules.
But one still has $\I^T(U,\L) \twoheadrightarrow_\G \I^T(U,\k)$, because $\p_\k$ splits totally in  $\L/\k$. By comparing the $p$-rank, one  finally obtains: $\I^T(U,\L) \simeq \I^T(U,\k) \simeq \fq_p$. 
\end{proof}

\medskip

To finish this part, we present a result of avoidance.

\begin{prop} \label{inertie} Suppose that the  $\fq_p[\G]$-module $\Gal(\F^T/\L')$ contains a free sub-module $\H'$ of rank $|\G|$  with basis $(x_g)_{g\in \G}$. Put $\displaystyle{x_0=\sum_{g\in \G} (g-1) x_g \in \H'}$. By Chebotarev density Theorem, take a prime ideal  $\P$ of $\O_\L$ such that $\langle \FF{\F^T/\L'}{\P}\rangle = \langle x_0\rangle $. 
Suppose moreover that $$\H' \bigcap \F(S)= \{0\},$$
where  $\F(S)$ is the subgroup of  $\Gal(\F^T/\L')$ generated by the  Frobenius of a $\G$-stable set $S$ of ideals of $\O_\L$.
Then, as $\G$-modules,  $\I_S^T(U,\L) \simeq \I_S^T(U,\k) \simeq \I^T(U,\k) \simeq \fq_p$, where $U=\{ \P^g, \ g\in \G\}$. 
Moreover $\I_S^T(U,\L) \cap \I_U^T(S,\L)=\{e\}$.
\end{prop}

\begin{proof} 
As  $x_0 \in \H'$,   the module $\langle x_0 \rangle_\G $ intersects   $\F(S)$ trivially. 
As for Proposition \ref{lemmepreparatoire}, it implies that any  $T$-split cyclic degree  $p$ extension  of $\L$,  $S$-ramified and totally ramified at $\P_0|\p$ is totally ramified at  all  $\P_0^g$, $g\in \G$. 
Hence,  $d_p \I_S^T(U,\L) \leq 1$. 
But by Proposition  \ref{lemmepreparatoire}, one knows that $d_p \I^T(U,\L) \geq 1$. As  $\I_S^T(U,\L) \twoheadrightarrow \I^T(U,\L)$
one obtains  $\I_S^T(U,\L) \simeq_\G \fq_p$.

 Suppose now $\I_S^T(U,\L) \cap \I_U^T(S,\L) \neq \{e\}$. As $\I_S^T(U,\L)$ is of order $p$, it implies that  $\I_S^T(U,\L) \subset \I_U^T(S,\L)$ and then every $T$-split, $U$-ramified, cyclic degree $p$ extension  of $\L$,  is in fact everywhere unramified, which contradicts   $\I^T(U,\L) \simeq \fq_p$.
\end{proof}

\begin{rema} 
The main question it to find an element $x$ in $\big(\V^T/(\L^\times)^p\big)^*$ such that ${\rm Ann}(x)=\fq_p N$. In some cases,  one can find a such element in a free module of rank  $1$. Typically if  $\G=\langle x \rangle$ is cyclic, it suffices to take  $x=g-1$. Or, in the semisimple case,  \emph{i.e.} when  the order of $\G$ is coprime to $p$, take  $\displaystyle{x=\sum_{g \in \G}(g-1)}$.   
\end{rema}


 

\part{Proof of the main results} \label{PartIII}

\section{The strategy}\label{section-proof}

\subsection{}

Let $\L/\K /\k$ be a $\sigma$-uniform tower; put $\Gamma=\Gal(\L/\K)$, $\G=\Gal(\L/\k)$ and $\Delta=\langle \sigma \rangle$. We still assume that $\sigma$ is of order $\ell~|~(p-1)$.

Denote by $d$ the $p$-rank of $\Gamma$ and by $r$ the $p$-rank  of the fixed points of $\sigma$ acting on $\Gamma^{\pel}=\Gamma/\Phi(\Gamma)$.  Let $x_1,\cdots , x_{n} \in \Gamma$ be some lifts of some generators of $\Gamma^{\pel}$ respecting the action of $\sigma$ (see \S \ref{semisimpleaction}). We fix $x_1\cdots, x_r$ the lifts of the fixed points. Hence, by Proposition \ref{pointfixe-uniform}, $\gamchapzero=\langle x_1,\cdots, x_r\rangle$, the pro-$p$ group $\gamchap$ is topologically generated by the conjuguates $x_i^g, i=1, \cdots, r$,  $g\in \GG:=\Gamma/\gamchap$  of the  $x_i$. Moreover 
by  proposition \ref{generateurs}, $\gamchappel$ is minimaly generated as $\fq_p[[G]]$-module by the family  $\{x_1 \Phi(\gamchap),\cdots, x_r\Phi(\gamchap)\}$.

\subsection{} \label{section-clef}
Now assume that $\Gamma$ is the Galois group of a pro-$p$ extension unramified outside $S$ and totally split at $T$, \emph{i.e.} a quotient of $\GG_S^T=\Gal(\K_S^T/\K)$. Suppose moreover that the places in $S$ are coprime to $p$, in other words, $S$ is tame. Then $\GG_S$ and $\Gamma$ are  {FAb}. Put $\F:=\L^{\gamchap}$ and $\GG:=\Gal(\F/\K)$. The situation is summarized in the diagram below.

$$
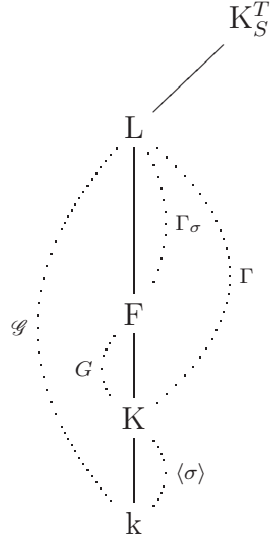
\xymatrix{& \K_S^T\\
 \L \ar@{-}[ur] \ar@{-}[dd]\ar@{.}@/^3pc/[ddd]^{\Gamma} \ar@{-}[dd]\ar@{.}@/^1pc/[dd]^{\gamchap} \ar@{.}@/_3pc/[dddd]_{\G} & \\
 &\\ 
 \F \ar@{-}[d]\ar@{.}@/_1pc/[d]_{G}
 &\\
 \K \ar@{-}[d]\ar@{.}@/^1pc/[d]^{\langle \sigma \rangle}& \\
\k &}$$

By Proposition \ref{fini1}: $[\F:\K]< \infty$, and by maximality of $\K_S^T$, one has $\K_S^T=\F_S^T$; put $\GG_S^T(\F):=\Gal(\K_S^T/\F)$. 

\medskip

Then the natural map $\GG_S^T(\F) \twoheadrightarrow \gamchap$ factors through $\psi:\GG_S^{T,ab}(\F) \twoheadrightarrow (\gamchap)^{ab} $.

Of course, $\GG$ acts on $\GG_S^T(\F)$ and on $\gamchap$ and then $\psi$ is a $\GG$-morphism of abelian groups.

\

We recall that $x_1\cdots, x_r $ are in $\Gamma $, they can be lifted to $\GG_S^T$. In fact, by construction, the  elements $x_1\cdots, x_r $ are in $\GG_S^T(\F)$  and by Proposition \ref{generateurs}, their classes generate   $\displaystyle{\gamchappel}$ as $\fq_p[\GG]$-modules. Put $\MM:=\langle \GG\cdot x_i \Phi(\GG_S^T(\F)), \ i=1,\cdots, r\rangle \subset \big(\GG_S(\F)\big)^{\pel}$.

\begin{prop}\label{analytic-arithmetic} 
The morphism  $\psi$ induces a surjective $\GG$-morphism from $\MM $ to $\displaystyle{\gamchappel}$.

\end{prop}

Now, we make our key observation: the group $\MM$ is a subgroup of   $\big(\GG_S^T(\F)\big)^{\pel}$, it may be described by class field theory, and the $G$-structure of $\displaystyle{\gamchappel}$  depends only on the pro-$p$ group $\Gamma$.

\medskip

{\it As we have mentioned  in the beginning of this work, the goal is  to find some situations where the $\GG$-structures of $\MM$ and of $\displaystyle{\Gamma_\sigma^{ab}}$ are not compatible.}

\section{Proof of Theorem \ref{maintheorem0}}

Let us start with a notation. For a finitely generated pro-$p$-group $G$, denote by $(G_n)_n$ the central series of $G$:   $G_0=G$, and for $n \geq 0$, $G_{n+1}=[G_n,G_n]$. Put   $G^{n}=G/G_n$.

\begin{defi}
Let $n\geq 1$. Denote by $\K_S^{(n)}$ the subfield of $\K_S$ fixed by $(\GG_S)_n$, whose Galois
group over $\K$ is thus $\GG_S^n$. It is also the $n$th step of the  $p$-tower  $\K_S/\K$ of $\K$, unramified outside $S$.
\end{defi}

Recall the integer $m(\ell)$ defined in (\ref{entier-m}): it is an upper bound of the solvability length of the quotient $G:=\Gamma/\gamchap$. See also remark \ref{longueur-derivee}.

\medskip

We are now able to prove Theorem \ref{maintheorem0} of the section  \ref{section-presentation}.

\begin{theo}[Theorem \ref{maintheorem0}]\label{maintheorem}  Let $\K$ be a number field  equipped with an automorphism $\sigma $ of order  $\ell~|~(p-1)$; put $\k=\K^\sigma$.
Suppose that there exists a finite set  $S$ of tame primes of  $\O_\k$ such that the action of  $\sigma$ on  $\GG_S^{ab}$ is fixed point free. Fix $s \in \Z_{>0}$.

Let $T$ be a finite set of prime ideals of $\O_\k$ that totally split  in  $\K^{(m(\ell))}_S/\k$ and such  that  $|T|\geq \A + s|\G|(|S||\G|+1)$, where $\A=\A(\K^{(m(\ell))}_S/\K) $  (see  remarks \ref{remaA0} and \ref{remaA1}). Then there exists  $s$ sets $\SS_1, \cdots, \SS_s$, of ideal primes of  $\O_\k$, all of positive density,   such that for $\Sigma=S\cup S'$ with  $S'=\{\p_1,\cdots, \p_s\}$, where   $\p_i \in \SS_i$,  $i=1,\cdots, s$, one has: 
\begin{enumerate}
\item[(i)]  $(\GG_\Sigma^{T})^{\pel} \simeq_\G (\GG_S )^{\pel} \bigoplus \ ( \fq_p)^{\oplus^s}$;
\item[(ii)] there is no continuous Galois representation $\rho : \GG_\Sigma^T \rightarrow \Gl_m(\Q_p)$  which is \fpm \ and $\Gamma_\sigma$ is supported at $S'$, where $\Gamma$ is the image of $\rho$.
\end{enumerate} 
\end{theo}

\begin{proof}
The proof is a combination of the previous results. 
First,  the extension  $\K^{(m(\ell))}/\k$ is a Galois extension. Put $\L_0=\K_S^{(m(\ell))}$ and $\G=\Gal(\L_0/\k)$. 
 Consider  the  $\fq_p[\G]$-module $\fq_p\otimes \O_{\L_0}^T$ and let  $\fq_p[\G]^t \oplus N$ be its decomposition  as $\fq_p[\G]$-modules where   $N$ is of torsion (see \S \ref{section5.1.3}). Thanks to Theorem  \ref{constante-liberte}, as $T$ is  sufficiently large, one gets    $t \geq s|\G|(|S||\G|+1)$.

\medskip

Let us conserve the notations of \S \ref{rappel-kummer}.
Let $\F(S)$  be the sub-$\fq_p[\G]$-module  of $\Gal(\F^T/\L')$ generated by the Frobenius of the prime ideals of $S$ (see \S \ref{chebotarev-consequences}).

\begin{lemm}\label{lemme-technique}
Suppose that $t\geq s|\G|(|S||\G|+1)$.  Then there exists  $s|\G|$  $T$-units $\varepsilon_g^i \in \O_{\L_0}^T$,  $g\in \G$, $i=1,\cdots, s$,  such that 
\begin{enumerate}
\item[(i)] for every  $i=1,\cdots, s$, the $\fq_p[\G]$-module   $\displaystyle{\sum_{g\in \G}\fq_p[\G] \varepsilon_g^i}$ is free of rank $|\G|$, with basis $\{\varepsilon_g^i,\ g\in \G\}$;
\item[(ii)] the $\fq_p[\G]$-modules  $\displaystyle{\sum_{g\in \G}\fq_p[\G] \varepsilon_g^i}$ are in direct factors: 

$\displaystyle{\sum_{i=1}^s \sum_{g\in \G}\fq_p[\G] \varepsilon_g^i=\bigoplus_{i=1}^s \big(\sum_{g\in \G}\fq_p[\G] \varepsilon_g^i\big) }$;
\item[(iii)] following the notations of  Section \ref{rappel-kummer},  for $i=1,\cdots, s$, $$\Theta\big(\big( \displaystyle{\sum_{g\in \G} \fq_p[\G] \varepsilon_g^i\big)^*(1)\big) \cap \F(S)=\{0\}}.$$ 
\end{enumerate}
\end{lemm}

\begin{proof}  
Let us start with $\fq_p\otimes \O_{\L_0}^T \simeq \fq_p[\G]^t \oplus N$. By Kummer duality, $\Gal(\F^T/\L')$ contains  $\fq_p[\G]^t$ in direct factor, the free modules coming from the image of the  dual of  $T$-units by $\Theta$ (see \S \ref{kummer0}). As $d_p \F(S) \leq |S||\G|$, the $\fq_p[\G]$-module $\F(S)$ intersects at most $|S||\G|$ modules each one isomorphic to $\fq_p[\G]^{s|\G|}$. Hence as $t\geq s|\G|(|S||\G|+1)$,  there exists at least one module isomorphic to $\fq_p[\G]^{s|\G|}$ that does not intersect
  $\F(S)$, in other words, there exists $s|\G|$ free submodules  $\MM_i$ of $\Gal(\F^T/\L')$, $i=1,\cdots, s|\G|$, all in direct factors, such that  $\displaystyle{\F(S) \cap \big( \sum_i \MM_i\big) =\{0\}}$. Then  the $T$-units given by $\Theta^{-1}(\MM_i)$ satisfy (i), (ii) and (iii) of the Lemma.
\end{proof}

 \medskip
 
 Let us adapt  the Proposition \ref{inertie} in our context. 
 For $i=1,\cdots, s$, let $\H^i $ be the free $\fq_p[\G]$-modules of basis  $\{x_{g}^i, \ g\in  \G\}$. Recall that these modules are obtained by Kummer duality from the $T$-units  of Lemma \ref{lemme-technique}. Put also $\displaystyle{x_0^i:= \sum_{g\in \G} (g-1)x_g^i \in \H'}$. 
 By Chebotarev density Theorem, let  $\SS_i$ be the set of prime ideals   $\p$ of $\O_\K$, such that the  (class of) Frobenius of $\p$ 
 in $\F^T/\k$ corresponds to  $x_0^i$: the  $\SS_i$ is of positive density. 
   
   \medskip
   
   Then consider $S'=\{\p_1,\cdots, \p_s\}$ a set of prime ideals of  $\O_\k$, with $\p_i \in \SS_i$;  put  $\Sigma= S\cup S'$.
   
   \medskip

For $i=1,\cdots, s$, choose a  prime ideal $\P_i|\p_i$ of $\O_{\L_0}$ above  $\p_i$. Put $U_i=\{\P_i^g, g\in \GG\}$.

\medskip

Let us fix $i\in\{1, \cdots, s\}$, and put $$S_i=S\cup U_1 \cup \cdots \cup U_{i-1}\cup U_{i+1} \cup \cdots \cup U_r,$$ here, we drop $U_i.$

\begin{lemm} \label{lemme-technique1}
\begin{enumerate}
\item[(i)] Let $\R'/\k$ be a Galois subextension of $\L_0/\k$ of Galois group $G'$.
Then  as  $\fq_p[G']$-modules: $$\I_S^T(S',\R') \simeq \bigoplus_{i=1}^s \I^T_{S_i}(U_i,\R' ) \simeq \big(\fq_p\big) ^{\oplus^s}.$$
\item[(ii)] At the level of $\K$,  one has: $$\big(\GG_\Sigma^{T}(\K)\big)^{\pel} \simeq_\G \big(\GG_S^T(\K)\big)^{\pel} \bigoplus \big(\fq_p\big)^{\oplus^{s}} \simeq_\G \big(\GG_S(\K)\big)^{\pel} \bigoplus \big(\fq_p\big)^{\oplus^{s}} \cdot$$
\end{enumerate}
\end{lemm}

\begin{proof} (i) First, take $R'=\L_0$ and fix $i$.
By Lemma \ref{lemme-technique}, $\H_i \cap \F(S_i)=\{0\}$.
The proposition \ref{inertie} applied to $U_i$ and to $S_i$ allows us  to get: $\I_{S_i}^T(U_i,\L_0)\simeq \I_{S_i}^T(U_i,\K) \simeq \I^T(U_i,\k) \simeq  \fq_p$ and $\I_{S_i}^T(U_i,\L_0)\cap \I_{U_i}^T(S_i,\L_0)=\{1\}$.
Hence when $i$ varies,  the groups  $\I^T_{S_i}(U_i,\L_0)$ are in direct factors  in $\big(\GG_\Sigma^T(\L_0)\big)^{\pel}$.

Take now $\R'$ in $\L_0/\k$. As $\L_0$ is $S'$-ramified, one has $\I_S^T(S',\L_0)\twoheadrightarrow  \I_S^T(S',\R')\twoheadrightarrow \I_S^T(S',\k)$ and one concludes thanks to  $\I_S^T(S',\L_0)\simeq  \I_S^T(S',\k)$.

(ii)  comes from the exact sequence of $\fq_p[\langle \sigma\rangle]$-modules (which splits by semisimplicity):
$$1\longrightarrow \bigoplus_{i=1}^s \I^T_{S_i}(U_i,\K) \longrightarrow \big(\GG_\Sigma^T(\K)\big)^{\pel} \longrightarrow \big(\GG_S^T(\K)\big)^{\pel} \longrightarrow 1 \cdot $$
and by  the choice of $T$: $\big(\GG_S^T(\K)\big)^{\pel} \simeq \big(\GG_S(\K)\big)^{\pel}$.
\end{proof}

\

 Let us start with a  $\sigma$-uniform extension $\L/\K/\k$ such that  $\Gal(\L/\K)$ is a uniform quotient of $\GG_\Sigma^T(\K)$.
Put $\Gamma=\Gal(\L/\K)$ and assume that  $d\geq 1$.

\

As $(\GG_\Sigma^T(\K))^{\pel} \twoheadrightarrow \Gamma^{\pel}$, the action of $\sigma$ on $\Gamma^{\pel}$ has at most $s$ "fixed points". Moreover by Boston \cite{Boston2} and \cite{Boston3}, this action must have at least one non-trivial fixed point. Hence, here  as $\Gamma$ is supposed to be non trivial, we get   $1 \leq r \leq s$, where  $r=\dim_{\fq_p} (\Gamma^{\pel})_\sigma$. Denote by $x_1,\cdots, x_{n} \in \Gamma $ the element of $\Gamma$ that respect the  action of $\sigma$, with the choice:  $\sigma(x_i)=x_i$, for $i=1,\cdots, r$ (see section \ref{semisimpleaction}).

By Lemma \ref{pointfixe-uniform} and Proposition \ref{generateurs}, one knows that $\gamchap$ is topologically generated by the  $\GG$-conjuguates of the $x_i$, $i=1,\cdots,r$, where $\GG:=\Gamma/\gamchap$.
 It is the notion of \fpm \ that will give us some information about the $x_i$, $i=1,\cdots, r$.
 
 \medskip
 
 For $i=1,\cdots, s$, and $\p_i \in S'$, let  $y_i$ be a generator of the inertia group $I_{\p_i}$ of a ideal prime $\p_i$ in $\L_0$: $I_{\p_i}=\langle y_i \rangle$.

 \medskip It is clear that $I_{\p_i}$ intersects non trivially $\gamchap$. 
The  \fpm  \ impose then $$\langle y_1,\cdots, y_{s} \rangle^{\rm Norm}=  \gamchap\cdot $$ In particular
\begin{enumerate}\item[(i)]  if we note by $\F$ the subfield of $\L$ fixed by $\gamchap$,  then the extension $\F/\K$ is unramified at $S'$,  $\F \subset \K_S^T$ and $\GG_S^T(\K) \twoheadrightarrow \Gal(\F/\K)$;
\item[(ii)] the pro-$p$ group $\gamchap$ is generated by  the $G$-conjuguates of the $y_i$, $i=1,\cdots, s$.
\end{enumerate}

\medskip

On the other hand, the action of $\sigma$ on $G=\Gamma/\gamchap$ is fixed point free, hence  $G$ is nilpotent of length at most $n(\ell)$ (see remark \ref{longueur-derivee}).  Consequently, we get:  $\F\subset (\K_S^{T})^{(m(\ell))}=\K_S^{(m(\ell))}=\L_0$ by the choice of  $T$.

\medskip

By Lemma \ref{lemme-technique1}, the $\fq_p$-vector space $\I_S^T(S',\F) $ is of dimension $s$ and the action of  $G:=\Gal(\F/\K)$  on it is trivial: indeed, $\I_S^T(S',\F)\simeq_\G (\fq_p)^{\oplus^s}$.  But, by  Proposition \ref{analytic-arithmetic} and by the condition above the ramification at the prime ideals  $\p_i \in S'$,   $\I_S^T(S',\F) \twoheadrightarrow (\gamchap)^{\pel}$ and then  $G$ acts trivially on $(\gamchap)^{\pel}$. At this point, one uses the condition \fpm \ to obtain a contradiction: indeed in this case $G$ should act non trivially on $\displaystyle{(\gamchap)^{\pel}}$!
\end{proof}

\medskip

 \section{Applications}

\subsection{When  $\sigma$ is of order $2$}

 Theorem \ref{maintheo1bis} gives a context where the condition about the ramification is automatically satisfied.  Lets us give  a proof.

 \medskip
 
 We still conserve the main notations of Theorem  \ref{maintheorem}: let   $\K/\k$ be a quadratic extension; put  $\Gal(\K/\k)=\langle\sigma \rangle$. Let  $S$ be a finite set of ideal primes of $\O_\k$ such that  $p\nmid |\Cl_S(\k)|$.

 Let $S'=\{\p_1,\cdots, \p_s\}$ be a finite set of prime ideals of  $\O_\k$ such that $$\displaystyle{(\GG_\Sigma(\K))^{\pel} \simeq   (\GG_S(\K))^{\pel} \bigoplus (\fq_p)^{\oplus^s}},$$ where $\Sigma=S\cup S'$, with a slight abuse of notation. Let  $\I(S')$ be the subgroup of $\GG_\Sigma^{ab}(\K)$ generated by the inertia groups of the primes in $S'$. One then has $1 \longrightarrow \I(S') \longrightarrow \GG_\Sigma^{ab}(\K) \longrightarrow \GG_S^{ab}(\K) \longrightarrow 1$.
 
 Take a minimal set of generators $\{x_1,\cdots, x_s, y_1,\cdots, y_t\}$ of $\GG_\Sigma^{ab}=\GG_\Sigma^{ab}(\K)$ as follows: the elements $x_1,\cdots,x_s$ satisfy $\sigma(x_i)=x_i^{-1}$ and the elements $y_1,\cdots, y_t $ satisfy  $\sigma(y_i)=y_i $ (see for example \cite{HR}, Theorem 2.3).

 \begin{lemm}\label{non-ram0}  Under the conditions of this section, one has $\I(S')=\langle y_1,\cdots, y_t\rangle$.
 \end{lemm} 
 
 \begin{proof} As $\I(S')\big(\GG_\Sigma^{ab}\big)^p/\big(\GG_\Sigma^{ab}\big)^p \simeq \ker \left( \GG_\Sigma^{\pel} \rightarrow \GG_S^{\pel}\right)$, 
  $$t = d_p \big[ \I(S')\big(\GG_\Sigma^{ab}\big)^p/\big(\GG_\Sigma^{ab}\big)^p\big] \leq d_p \I(S') \leq |S'|=t,$$
 hence, $d_p \I(S')=|S'|$.
 
 \medskip 
 As $\sigma$ acts by $-1$ on $\GG_S^{ab}$, we get $\langle y_1,\cdots, y_t\rangle \subset \I(S')$. Suppose  $\langle y_1,\cdots, y_t\rangle \subsetneq \I(S')$. Put  $x\in \I(S') \backslash \langle y_1,\cdots, y_t\rangle$. Then
 there exists $ y \in \langle y_1,\cdots, y_t\rangle$ such that $xy \in \langle x_1,\cdots, x_s \rangle$ 
 with $xy \neq e$.
 As $\GG_\Sigma^{ab}=\langle y_1,\cdots, y_t\rangle \oplus \langle x_1,\cdots, x_s \rangle$, one gets $d_p \I(S') \geq d_p \langle y_1,\cdots, y_t\rangle +1$, and so a contradiction. 
 \end{proof}

    Let  $\L/\K/\k$ be a $\sigma$-uniform tower in  $\K_\Sigma/\k$. Put $\F:=\L^{\gamchap}$.
   
   Let us recall that $\Gal(\F/\K)$ is fixed point free under the action of $\sigma$ of order $2$: hence $\F/\K$
   is an abelian subextension of $\K_\Sigma^{ab}$.

  \begin{lemm} \label{non-ram}
The extension $\F/\K$ is  $S$-ramified. Moreover, $\F=\L\cap \K_S^{ab} $. \end{lemm}

   \begin{proof} By Lemma \ref{non-ram0},  the involution $\sigma$ acts trivially on $\I(S')$.  As $\sigma$ acts without non-trivial fixed point fixed on $G=\Gamma/\gamchap$ and that  $G_\Sigma^{ab} \stackrel{\theta}{\twoheadrightarrow} G$, one then gets $\theta(\I(S'))=\{1\}$, meaning exactly that  $\F/\K$ is $S$-ramified, \emph{i.e.} $\F \subset \K_S^{ab}$.
   Put $\F_1= \L\cap \K_S^{ab} $. Obviously, $\F \subset \F_1$. As  $\sigma$ acts by  $-1$ on $\Cl_S(\K)$, 
   $\sigma$ acts by $-1$ on     $\Gal(\F_1/\F)$: indeed if not, $\Gal(\K_S^{ab}/\F)$ would have a fixed point (see the proof of Proposition \ref{fini1}). On the other hand, as $\F_1/\K$ is abelian, one still has $\big(\Gamma_\sigma^{ab}\big)_G \twoheadrightarrow \Gal(\F_1/\F)$. But by  Proposition \ref{generateurs}, the involution   $\sigma$ acts trivially on $\big(\Gamma_\sigma^{ab}\big)_G$, which implies that  $\sigma$ acts trivially on $\Gal(\F_1/\F)$.  
To conclude: $\sigma$ acts at a time by $-1$ and by $+1$ on $\Gal(\F_1/\F)$, consequently  $\F_1=\F$.
   \end{proof}

 \begin{rema}\label{remarque-ramification}
 Lemma \ref{non-ram} shows that the inertia groups  of the prime ideals $\p \in S_0$ are in  $\gamchap$.
  \end{rema}

\begin{prop}\label{ramification-pointfixe} Let us conserve the notations and the conditions of this section.
By Chebotarev density Theorem, choose a finite set $T$ of prime ideals of $\O_\k$, disjoint from $S$, such that:
\begin{itemize}
\item each prime ideal of $T$ totally splits in  $\K_S^{ab}/\K$;
\item $\Cl_S^T(\K_S^{ab})$ is trivial.
\end{itemize} 
Let $\rho : \GG_\Sigma^T \rightarrow \Gl_m(\Q_p)$ be a continuous representation 
with  $\sigma$-uniform image $\Gamma$. Then   $\gamchap$ is supported at  $S'$, meaning the inertia groups  of the prime ideals of $S'$ generate  the  group $\gamchap$.
\end{prop}

\begin{proof} The $\sigma$-uniform tower $\L/\K/\k$ is in  $\K_\Sigma^T/\k$ and then in $\K_\Sigma/\k$: 
one can apply Lemma  \ref{non-ram} to this situation.
By Lemma \ref{non-ram} the inertia groups  of $\p\in S'$ are in $\gamchap$. Denote by  $\L_1$ the subfield of $\L$  fixed by these  inertia groups: the extension $\L_1/\F$ is  $T$-split and $S$-ramified. Suppose that $\L_1/\F$  is not trivial. Then  one can assume that  $\L_1/\F$ is of degree $p$. Then by  Lemma \ref{non-ram}, we get that $\L_1\K_S^{ab}/\K_S^{ab}$ is $T$-split and $S$-ramified, cyclic degree $p$ extension. But by hypothesis $\Cl_S^T(\K_S^{ab})$  is trivial, and then, by class field theory, one obtains a contradiction. 
\end{proof}

 \medskip
 
 We can now say few words about the proofs of the results of  \S \ref{section-introduction} and \S \ref{section-presentation}.

  \begin{itemize}
  \item
  Theorem \ref{maintheo1bis} can be deduced from Theorem \ref{maintheorem} and from Proposition \ref{ramification-pointfixe}.
 \item
 Theorem of the subsection \ref{rappel-boston} comes from the fact that   every involution $\sigma$ on $\Sl_2^1(\Z_p)$ is of type $t_\sigma(\Gamma)=(1,b)$ and then is  {\fpm}  
 by Proposition \ref{prop-sl2}. (Here $T$ sufficiently large  means  also  that $\Cl^T(\K^H))$ is trivial.)
 \item Corollary  \ref{coro1} comes from the fact that the action of  $\sigma$ on $\Gamma$ should be trivial.  Thus  $\Im(\rho)$ comes  from $\k$ by compositum and then it suffices to remark that $d_p\Cl_S(\k) \leq |S|$. (Here, as previous,  $T$ sufficiently large  means also that  $\Cl^T(\K^H))$ is  trivial.)
\item Corollary \ref{coro-sln} can be deduced from Theorem \ref{maintheo1bis} and Proposition \ref{prop-sln}.
 \end{itemize}

\medskip

 \subsection{Along a  $\Z_p$-extension}
 The context of  the cyclotomic $\Z_p$-extension allows one to take    $T$ as small as possible.
 
\subsubsection{When $\ell=2$} \label{casl=2}  Take $p>2$.  Let  $\K/\k$ be a quadratic extension  such that: 
\begin{itemize}\item[(i)] $\K$ is totally real. Put $r_1=[\k:\Q]$;
\item[(ii)]  the $p$-class group along the $\Z_p$-cyclotomic extension of $\k$ is trivial;
\item[(iii)]  the number field $\K$ satisfies the Greenberg's conjecture.
\end{itemize}

 For $n\geq 0$, put $\K_n$ (resp. $\k_n$) for the $n$th steps  of  the $\Z_p$-cyclotomic extension $\K_\infty$  of $\K$ (resp. of $\k$): $[\K_n:\K]=[\k_n:\k]=p^n$. 
 
 \medskip

 We are going to apply Theorem \ref{maintheo1bis}  to the  $\sigma$-uniform extensions of $\K_n/\k_n$.

 \medskip
 
Take $n_0$ sufficiently large such that
\begin{itemize}
\item for all $n\geq n_0$, $\Cl(\K_{n+1}) \simeq \Cl(\K_n)$, which is always possible by condition (iii);
\item all prime ideals above $p$ are totally ramified  in  $\k_\infty/\k_{n_0}$.
\end{itemize}

  \medskip
  
  Let us fix $s\in \Z_{>0}$.

  \medskip
  
Put  $\C_0\simeq  \Cl(\K_n)$ and   $\C:=\Gal(\L_{n+1}/\k_n) \simeq \big( \C_0 \times \Z/p\Z\big) \rtimes \langle \sigma\rangle$. Let us apply the strategy developped in part \ref{PartII} in order  to find   free $\fq_p[\C]$-modules in $\O_{\L_n}^\times/(\O_{\L_n}^\times)^p$. Let us write $\O_{\L_n}^\times/(\O_{\L_n}^\times)^p=\fq_p[\C]^{t_n}\oplus N_n$, with $N_n$ of  torsion. Following  Theorem \ref{liber}, we get $t_n \geq r_1 p^n -(|\C|-1)d_p C -1$. Hence for large $n$,  we are guarantee that  $t_n \geq s|\C|$, and then the method developed  in the proof of Theorem \ref{maintheorem} can apply with  $T=\emptyset$ ! 
 Hence, there exists $s$ set  $\SS_i$, $i=1,\dots, s$  of prime ideals  of  $\O_{\k_n}$, all of positive density, such that
   for all set $S=\{\p_1,\cdots,\p_s\}$, with  $\p_i\in \SS_i$, one gets:
   $$\big(\GG_S(\K_{n+1})\big)^{\pel} \simeq \big(\GG_S(\K_{n})\big)^{\pel} \simeq  \C_0/p \bigoplus \ \big(\fq_p\big)^{\oplus^s} \cdot$$
  Consequently the  groups  $(\GG_S(\K_n)^{\pel})_n$  stabilize  in two consecutive steps: by a classical argument in Iwasawa theory (see for example \cite{fukuda}, theorem 1), one obtains that for  $n\geq m$, with $m$ sufficiently large:   $\GG_S(\K_n)^{\pel} \simeq \C_0/p \bigoplus \ \big(\fq_p\big)^{\oplus^s}$. Applying the strategy of the proof of Theorem \ref{maintheo1bis}, one obtains the following corollary:

\begin{coro}[Theorem \ref{cyclo}]
Under the conditions of this section, for sufficiently  large $m \in \Z_{>0}$, there exists $s$ set  $\SS_i$, $i=1,\cdots, s$, of prime ideals of $\O_{\k_m}$, all of positive density,  such that for all  set $S=\{\p_1,\cdots,\p_r\}$, with $\p_i \in \SS_i$, and for all $n\geq m$, one has: 
\begin{enumerate}
\item[(i)] $\big(\GG_S(\K_n)\big)^{\pel} \simeq  \C_0/p \bigoplus \ \big(\fq_p\big)^{\oplus^s}$ ;
\item[(ii)] the non existence of continuous representation $\rho : \GG_S(\K_n) \rightarrow \Gl_m(\Q_p)$ with  $\sigma$-uniform image  \fpm \ and $\gamchap$ supported at $S$. 
\end{enumerate}\end{coro}

  One can say more. Indeed, let us choose moreover a set 
 $T$  of prime ideals of  $\k_{m}$, such that
 \begin{itemize}
 \item each ideal prime of $T$ splits totally in  $\K_{m+1}^{(1)}/\k_{m}$;
 \item  $\Cl^T(\K_{m+1}^{(1)})$ is trivial.
 \end{itemize}
 Then,   $\Cl^T(\K_n^{(1)})$ is trivial for all $n\geq m$. 
  Proposition \ref{ramification-pointfixe} shows   that the  ramification will be supported by the fixed points.
  
  \begin{coro} With the conditions and notations of this section, for  $m$ sufficiently large,
  there exists $s$ sets  $\SS_i$, $i=1,\cdots, s$, of prime ideals of $\O_{\k_m}$, all of positive density, 
  such that for all set $S=\{\p_1,\cdots,\p_r\}$, with $\p_i \in \SS_i$, and for all $n\geq m$, one gets: 
\begin{enumerate}
\item[(i)] $\big(\GG_S^T(\K_n)\big)^{\pel} \simeq  \C_0/p \bigoplus \ \big(\fq_p\big)^{\oplus^s}$;
\item[(ii)] the non  existence of continuous representation $\rho : \GG_S^T(\K_n) \rightarrow \Gl_m(\Q_p)$ with  $\sigma$-uniform image \fpm. 
\end{enumerate}\end{coro}

 \

  To conclude this section, let us give an example. 
  
  Take $p=3$ and $\K=\Q(\sqrt{32009})$.  
   
     Let  $\displaystyle{\K_\infty=\bigcup_n\K_n}$  be the $\Z_p$-
   cyclotomic extension of $\K$. Put $\Gal(\K/\Q)=\langle \sigma \rangle$.
   A computation with Pari-GP \cite{pari} shows that for all  $n\geq 1$, $\Cl(\K_n) \simeq \Z/9\Z \times \Z/3\Z$.
  Following  Theorem \ref{constante-liberte},  remark \ref{remaA0} and Theorem \ref{maintheorem}, take  $n$ such that 
   $$r_1 3^n -(|C|-1)d_3 C -1 \geq s|C|,$$
   where $|C|=2\times 3^4$, $r_1=2$ and where $s$ is the number of fixed points that we want to introduce. Hence  $n\geq n_0= \lceil \log_3(2+s) +4\rceil$ holds. 
If moreover we take a set  $T$ of ideal primes of  $\O_{\k_{n_0}}$ all totally splits in $\K_{n_0+1}^H/\Q_{n_0}$ and such that  $\Cl^T(\K_{n_0+1})$ is trivial, one then gets:

\begin{coro}\label{exemple-cyclo}
Let  $\K=\Q(\sqrt{32009})$ and let $s\in \Z_{>0}$. Take $T$  as  before. There exists  $s$ sets $\SS_i$, $i=1,\cdots, s$,of prime ideals of $\O_{\Q_{n_0}}$, all of positive density, such that for all set $S=\{\p_1,\cdots, \p_s\}$, with $\p_i \in \SS_i$, and for $n\geq \log_3(2+s) +4$, one has:
\begin{enumerate}
\item[(i)] $\displaystyle{\GG_S^T(\K_n)^{\pel}}$ has $s$ independant fixed points under the action of $\sigma$;
\item[(ii)] there exists no continuous representation $\rho : \GG_S^T(\K_n) \longrightarrow \Gl_m(\Q_p)$ with  $\sigma$-uniform image  \fpm.
\end{enumerate} 
\end{coro}

\begin{rema}
If we start with a situation where the $p$-class group is cyclic along the $\Z_p$-cyclotomic extension, then  $\Cl(\K_n^H)$ is trivial: and then one can take   $T=\emptyset$. But in this case, the  group $\GG_S$ is of type  $(1,b)$, and one has seen in Proposition \ref{calcul-} (or Corollary \ref{coro-calcul-}) that this type is not compatible  with the type of   {FAb} uniform groups. And in this case, the expected conclusion is obvious!
\end{rema}

  \subsubsection{When $\ell$ is odd}
   
Let $\K/\k$ be a cyclic extension of  prime degree $\ell >2$. Assume that $\ell~|~(p-1)$.  Suppose that 
\begin{enumerate}\item[(i)] the extension  $\K/\k$ is totally real;
\item[(ii)]  the $p$-class group along the $\Z_p$-cyclotomic extension of $\k$ is trivial.
\end{enumerate}

Let us take the notation of the beginning of section  \ref{section-proof}. One has seen that $\K^{(m)}$ is the key number field, where $\K^{(m)}$is the  $m$th step of the Hilbert $p$-class field tower of  $\K$  and where $m = \log_2(n(\ell)+1)$.

Thus by the Greenberg's conjecture, for $n_0 \gg 0$ and for  $n\geq n_0$, on has $[\K_{n_0}^{(1)}:\K_{n_0}]=[\K_{n}^{(1)}:\K_{n_0}]$. If moreover, one assumes the Greenberg's conjecture for all the fields $\K_n$, there exists an integer  $n_1\geq n_0$ such that $n\geq n_1$, $[(\K_{n_1})^{(1)}:\K_{n_1}^{(1)}]=[(\K_{n}^1)^{(1)}:\K_n^{(1)}]$  and then $[\K_n^{(2)}:\K_n]=[\K_{n_1}^{(2)}:\K_{n_1}]$. By 
following this  process, one gets the existence of  $n_m \in \Z_{>0}$ such that for all  $n\geq n_m$, we get
   $[\K_n^{(m)}:\K_n]=[\K_{n_m}^{(m)}:\K_{n_m}]$ when supposing  the Greenberg's conjecture for the number fields  $\K_{n_i}$, $i=0,\cdots, m$.
      One can then apply the strategy of  the section \ref{casl=2} to obtain: 
  
  \begin{coro}
Under the conditions of this section,  in particular by assuming the Greenberg's conjecture for totally real number fields, for sufficiently  large $m \in \Z_{>0}$, there exists $s$ sets  $\SS_i$, $i=1,\cdots, s$ of prime ideals of  $\O_{\k_m}$, all of positive density, such that  for all set $S=\{\p_1,\cdots,\p_r\}$, with $\p_i \in \SS_i$, and for all $n\geq m$, one has:
\begin{enumerate}
\item[(i)] $\big(\GG_S(\K_n)\big)^{\pel}$ has $s$ independant fixed points under the action of $\sigma$; 
\item[(ii)] there is no continuous representation $\rho : \GG_S(\K_n) \rightarrow \Gl_r(\Q_p)$ of  $\sigma$-uniform image  
\fpm \ and where $\gamchap$ is supported at $S$.  
\end{enumerate}\end{coro}


\end{document}